\documentclass[11pt]{amsart}

\usepackage{amssymb}
\usepackage[all]{xy}
\usepackage[colorlinks=true, urlcolor=rltblue, citecolor=drkgreen, linkcolor=drkred] {hyperref}
\usepackage{xcolor}
\usepackage{pdfsync}	
\definecolor{rltblue}{rgb}{0,0,0.4}
\definecolor{drkred}{rgb}{0.6,0,0}
\definecolor{drkgreen}{rgb}{0,0.4,0}
\usepackage{graphicx}
\usepackage[mathscr]{eucal}

\usepackage{lmodern}
\usepackage[capitalize]{cleveref}
\usepackage{amssymb,amsmath,amsthm}
\usepackage{mathtools, MnSymbol}
\usepackage{setspace}
\usepackage{tikz-cd}
\usepackage[T1]{fontenc}
\usepackage[utf8]{inputenc}
\usepackage{textcomp} % provide euro and other symbols
\usepackage{stmaryrd}
\usepackage{datetime}
\usepackage{todonotes}
\usepackage{mdframed}
\usepackage[all]{xy}
\usepackage{todonotes}

\usepackage{thm-restate}

\usepackage{thmtools}
\usepackage{enumitem}

\declaretheorem[numberwithin=section]{theorem}
\declaretheorem[sibling=theorem]{lemma}
\declaretheorem[sibling=theorem]{conjecture}
\declaretheorem[sibling=theorem]{proposition}
\declaretheorem[sibling=theorem]{corollary}
\declaretheorem[sibling=theorem]{definition}
\declaretheorem[sibling=theorem]{question}
\declaretheorem[numberwithin=theorem]{claim}
\declaretheorem[numbered=no,name=Claim]{claimstar}
\newcommand{\SR}{\text{SR}}

\renewcommand{\L}{\mathcal{L}}
\newcommand{\mc}[1]{\mathcal{#1}}

\DeclareMathOperator{\Inf}{Inf}

\renewcommand{\phi}{\varphi}
\newcommand{\bigwwedge}{%
	\mathop{
		\mathchoice{\bigwedge\mkern-15mu\bigwedge}
		{\bigwedge\mkern-12.5mu\bigwedge}
		{\bigwedge\mkern-12.5mu\bigwedge}
		{\bigwedge\mkern-11mu\bigwedge}
	}
}

\newcommand{\bigvvee}{%
	\mathop{
		\mathchoice{\bigvee\mkern-15mu\bigvee}
		{\bigvee\mkern-12.5mu\bigvee}
		{\bigvee\mkern-12.5mu\bigvee}
		{\bigvee\mkern-11mu\bigvee}
	}
}
\newmdtheoremenv[backgroundcolor=cyan]{theorem-prove}{Theorem}[theorem]
\newmdtheoremenv[backgroundcolor=cyan]{lemma-prove}{Lemma}[theorem]
\newmdtheoremenv[backgroundcolor=cyan]{proposition-prove}{Proposition}[theorem]

\newmdtheoremenv[backgroundcolor=yellow!40]{theorem-check}{Theorem}[theorem]
\newmdtheoremenv[backgroundcolor=yellow!40]{lemma-check}{Lemma}[theorem]
\newmdtheoremenv[backgroundcolor=yellow!40]{proposition-check}{Proposition}[theorem]

\DeclareMathOperator{\fin}{fin}

\def\hbar{{\bar{h}}}

\def\om{\omega}

\def\L{\mathcal L}

\def\Lomom{\L_{\om_1,\om}}

\newcommand{\newmacro}[1]{\mathfrak{#1}}
\newcommand{\A}{\newmacro{A}}
\newcommand{\E}{\newmacro{E}}
\newcommand{\vwE}{\overline{\E}}
\newcommand{\vwA}{\overline{\A}}

\newtheorem{sublemma}[claim]{Sublemma}

\newcommand{\bfSigma}{\mathbf{\Sigma}}
\newcommand{\bfPi}{\mathbf{\Pi}}
\newcommand{\bfDelta}{\mathbf{\Delta}}

\DeclareMathOperator{\Mod}{Mod}

\def\la{\mathopen{\langle}}
\def\ra{\mathclose{\rangle}}

\def\and{\mathrel{\&}}

\protected\def\All{{\mathfrak{A}}}
\protected\def\Exs{{\mathfrak{E}}}
\let\hyphen=\-
\protected\def\-{\relax\ifmmode\expandafter\overline\else\expandafter\hyphen\fi}
\let\hash=\#
\let\#=\mathbb
\let\atsf=\@
\protected\def\@{\relax\ifmmode\expandafter\mathcal\else\expandafter\atsf\fi}

\let\bigwwedge=\bigdoublewedge
\let\bigvvee=\bigdoublevee

\makeatletter
\newmdenv[{
	linecolor=\@todonotes@bordercolor,
	backgroundcolor=\@todonotes@backgroundcolor,
	fontcolor=\@todonotes@textcolor,
	roundcorner=4pt,
}]{todobox}
\makeatother

\numberwithin{figure}{section}
\crefname{figure}{Figure}{Figures}
\setlist{font=\textup}

\usepackage{comment}
\usepackage{microtype}

\begin{document}
	
	\title{Optimal Syntactic Definitions of Back-and-Forth Types}
	\author{Ruiyuan Chen}
	\address[Chen]{
University of Florida\\
Department of Mathematics\\
1400 Stadium Rd\\
Gainesville, FL 32611, USA}
	\email{ruiyuan.chen@ufl.edu}
	\author{David Gonzalez} 
	\address[Gonzalez]{University of Notre Dame\\
Department of Mathematics\\
Hurley Hall, 255 Hurley, Notre Dame, IN 46556\\
  USA}
	\email{dgonza42@nd.edu}
	\author{Matthew Harrison-Trainor}
	\address[Harrison-Trainor]{University of Illinois Chicago\\
	Department of Mathematics, Statistics, and Computer Science\\
	851 S Morgan St, Chicago, IL 60607\\
	USA
}
	\email{mht@uic.edu}
	\thanks{The first author acknowledges support from the National Science Foundation under Grant No.\ \mbox{DMS-2224709}. The third author acknowledges support from a Sloan Research Fellowship and from the National Science Foundation under Grant No.\ \mbox{DMS-2153823}.}
	\subjclass{03C15, 03C57, 03D45, 03E15, 	03C75}
	
	\begin{abstract}
		The back-and-forth relations $\mathcal{M} \leq_\alpha \mathcal{N}$ are central to computable structure theory and countable model theory. It is well-known that the relation $\{(\mathcal{M},\mathcal{N}) : \mathcal{M} \leq_\alpha \mathcal{N}\}$ is (lightface) $\Pi^0_{2\alpha}$. We show that this is optimal as the set is $\mathbf{\Pi}^0_{2\alpha}$-complete. We are also interested in the one-sided relations $\{ \mathcal{N} : \mathcal{M} \leq_\alpha \mathcal{N}\}$ and $\{ \mathcal{N} : \mathcal{M} \geq_\alpha \mathcal{N}\}$ for a fixed $\mathcal{M}$, measuring the $\Pi_\alpha$ and $\Sigma_\alpha$ types of $\mathcal{M}$. We show that these sets are always $\mathbf{\Pi}^0_{\alpha + 2}$ and $\mathbf{\Pi}^0_{\alpha+3}$ respectively, and that for most $\alpha$ there are structures $\mathcal{M}$ for which these relations are complete at that level. In particular, there are structures $\mathcal{M}$ such that there is no $\Pi_\alpha$ (or even $\Pi_{\alpha+1})$ sentence $\varphi$ such that $\mathcal{N} \models \varphi \Longleftrightarrow \mathcal{M} \leq_\alpha \mathcal{N}$.
		
		This is unfortunate as not all $\Pi_{\alpha+2}$ sentences are preserved under $\leq_\alpha$. We define a new hierarchy of syntactic complexity closely related to the back-and-forth game, which can both define the back-and-forth types as well as be preserved by them. These hierarchies of formulas have already been useful in certain Henkin constructions, one of which we give in this paper, and another previously used by Gonzalez and Harrison-Trainor to show that every $\Pi_\alpha$ theory of linear orders has a model with Scott rank at most $\alpha+3$.
	\end{abstract}
	
	\maketitle
	
	\section{Introduction}
	
	Originally used to prove that any two countable dense linear orders without endpoints are isomorphic, the back-and-forth method is a classical tool in mathematical logic dating back over one hundred years. A general viewpoint, beginning with Ehrenfeucht \cite{Ehrenfeucht} and Fra\"iss\'e \cite{Fraisse}, is that by playing back-and-forth games between two first-order structures $\mc{M}$ and $\mc{N}$ we can capture which sets of formulas they agree on. Generally two players, \textsf{Spoiler} and \textsf{Duplicator}, take turns playing elements from the two structures, with \textsf{Duplicator} trying to maintain atomic facts between the elements in $\mc{M}$ and the elements in $\mc{N}$. In the Ehrenfeucht--Fra\"iss\'e game most commonly used in model theory, both players play single elements, and (in a finite language) the second player wins the game of length $n$ if and only if the two structures $\mc{M}$ and $\mc{N}$ satisfy the same sentences with quantifier depth $n$. Other versions, such as pebble games, can be used for finite variable logics.
	
	In this paper, we work with the back-and-forth game that corresponds to infinitary logic, and our main goal is to understand the complexity of deciding which player has a winning strategy. The key features of this particular back-and-forth game are that (a) on their turn, players play tuples of arbitrary size with the size chosen by \textsf{Spoiler}, (b) \textsf{Spoiler} must alternate between playing in $\mc{M}$ and playing in $\mc{N}$, with \textsf{Duplicator} alternating between playing in $\mc{N}$ and playing in $\mc{M}$, and (c) the games are played over a possibly infinite ordinal clock.
	
	This back-and-forth game yields the back-and-forth relations. (For simplicity, we assume that all languages in this paper are relational.)
	\begin{definition}\label{def:bfasym}
		The \textit{standard asymmetric back-and-forth relations} $\leq_\alpha$, for a countable ordinal $\alpha < \omega_1$, are defined by:
		\begin{itemize}
			\item $(\mc{M},\bar{a}) \leq_0 (\mc{N},\bar{b})$ if $\bar{a}$ and $\bar{b}$ satisfy the same quantifier-free formulas from among the first $|\bar{a}|$-many formulas.
			\item For $\alpha > 0$, $(\mc{M},\bar{a}) \leq_\alpha (\mc{N},\bar{b})$ if for each $\beta < \alpha$ and $\bar{d} \in \mc{N}$ there is $\bar{c} \in \mc{M}$ such that $(\mc{N},\bar{b} \bar{d}) \leq_\beta (\mc{M},\bar{a} \bar{c})$.
		\end{itemize}
		We define $\bar{a} \equiv_\alpha \bar{b}$ if $\bar{a} \leq_\alpha \bar{b}$ and $\bar{b} \leq_\alpha \bar{a}$.
	\end{definition}
	\noindent The interpretation of $(\mc{M},\bar{a}) \leq_\alpha (\mc{N},\bar{b})$ is that in the back-and-forth game between $\mc{M}$ and $\mc{N}$, starting with the partial isomorphism $\bar{a} \mapsto \bar{b}$ and with the first player \textsf{Spoiler} to play next in $\mc{N}$, the second player \textsf{Duplicator} can play without losing along an ordinal clock $\alpha$.
	The content of the classical back-and-forth argument is then that if \textsf{Duplicator} can continue playing forever, and $\mc{M}$ and $\mc{N}$ are countable, then $\mc{M} \cong \mc{N}$.
	
	The back-and-forth relations are a measure of similarity between $(\mc{M},\bar{a})$ and $(\mc{N},\bar{b})$. For example, for Boolean algebras $A$ and $B$:
	\begin{enumerate}
		\item $A \leq_1 B$ means that $A$ has at least as many elements as $B$;
		\item $A \leq_2 B$ means that $A$ and $B$ have the same number of elements and that $A$ has at least as many atoms as $B$;
		\item $A \leq_n B$ for $n \geq 3$ is determined in general by how $A$ and $B$ split up into atoms, atomless subalgebras, $1$-atoms, and so on. See \cite{HarrisMontalban}.
	\end{enumerate}
	Back-and-forth types have been central in the computability-theoretic study of Boolean algebras and particularly in the study of low$_n$ Boolean algebras \cite{DJ,Thurber,KnightStob,HM2}. For linear orders, $L \leq_1 M$ means that $M$ has at least as many elements as $L$, and $L \leq_2 M$ is more complicated, but Montalb\'an \cite{MLO} gives a complete description of $\equiv_2$ for linear orders. For some particular linear orders the back-and-forth relations are easier to describe, e.g., in \cite{AshKnight} it is shown that for $\gamma \geq 1$, $\omega^\beta \leq_{2 \gamma} \omega^\alpha$ if and only if $\alpha,\beta \geq \gamma$ or $\alpha = \beta$. For algebraic fields $\mathbb{Q} \subseteq E,F \subseteq \overline{\mathbb{Q}}$, $E \leq_1 F$ means that every polynomial $p(z) \in \mathbb{Q}[x]$ which as a root in $F$ has a root in $E$, and $E \equiv_1 F$ means that $E$ and $F$ are isomorphic which reflects the fact that it is easy to classify algebraic fields up to isomorphism. For general fields $E$ and $F$ the back-and-forth relations are harder to describe, but for example $E \leq_2 F$ implies that every finitely generated subfield of $F$ is (isomorphic to) a subfield of $E$. For every $\alpha$, there are fields $E$ and $F$ such that $E \equiv_\alpha F$ but $E \ncong F$.\footnote{This is a consequence of the Borel completeness of fields, see \cite{FS89} or \cite{MPSS}.} 
	
	These back-and-forth relations are related to the infinitary logic $\mc{L}_{\omega_1 \omega}$ which extends the standard finitary first-order logic by allowing countably infinite disjunctions and conjunctions. The formulas of $\mc{L}_{\omega_1 \omega}$ are stratified in terms of their complexity as measured by counting the number of alternations between existential and universal quantifiers.
	For an ordinal $\alpha < \omega_1$:
	\begin{itemize}
		\item A formula is $\Sigma_0$ and $\Pi_0$ if it is finitary quantifier-free.
		\item A formula is $\Sigma_{\alpha}$ if it is of the form
		\[ \bigdoublevee_i \exists \bar{x}_i \psi_i(\bar{x}_i)\]
		where each $\psi_i$ is $\Pi_{\beta}$ for some $\beta < \alpha$.
		\item A formula is $\Pi_{\alpha}$ if it is of the form
		\[ \bigdoublewedge_i \forall \bar{x}_i \psi_i(\bar{x}_i)\]
		where each $\psi_i$ is $\Pi_{\beta}$ for some $\beta < \alpha$.
	\end{itemize}
	The back-and-forth relations $\leq_\alpha$ characterize when two structures $\mc{M}$ and $\mc{N}$ satisfy the same $\Sigma_\alpha$/$\Pi_\alpha$ sentences as in the following theorem of Karp \cite{Karp} (see, e.g., \cite{MonBook}, Theorem II.36).
	
	\begin{restatable}[Karp]{theorem}{thmKarp}\label{thm:Karp}
		For any non-zero ordinal $\alpha$, structures $\mc{M}$ and $\mc{N}$ and tuples $\bar{a}\in\mc{M}$ and $\bar{b}\in\mc{N}$, the following are equivalent:
		\begin{enumerate}
			\item $(\mc{M},\bar{a})\leq_\alpha (\mc{N},\bar{b})$.
			\item Every $\Pi_\alpha$ formula true of $\bar{a}$ in $\mc{M}$ is true of $\bar{b}$ in $\mc{N}$.
			\item Every $\Sigma_\alpha$ formula true of $\bar{b}$ in $\mc{N}$ is true of $\bar{a}$ in $\mc{M}$.
		\end{enumerate} 
	\end{restatable}
	\noindent Both these back-and-forth relations and the infinitary logic $\mc{L}_{\omega_1 \omega}$ have proven to be very useful in countable structure theory, the study of countable structures using tools from model theory, descriptive set theory, and computability theory. See, e.g., the standard texts \cite{AshKnight} and \cite{MonBook1,MonBook} where they appear throughout.
	
	\medskip
	This paper is about the complexity of the relations $\mc{M} \leq_\alpha \mc{N}$ for different $\alpha$. One might hope, because the $\alpha$-back-and-forth relations are about preservation of $\Sigma_\alpha$/$\Pi_\alpha$ formulas, that there would be a characterization which is also on the level of $\alpha$-many alternations of quantifiers. While interesting in its own right, this question is motivated by certain appearances of the back-and-forth relations in arguments in countable model theory. Before giving our main results, we will talk about some of the motivation behind this question to set the stage.
	
	Our main motivation has to do with Scott's isomorphism theorem. Given a countable $\mc{L}$-structure $\mc{M}$, the set of $\mc{L}$-structures isomorphic to $\mc{M}$ is naively an analytic set in the Polish space $\Mod(\mc{L})$ of countably infinite $\mc{L}$-structures. It is a surprising result of Scott \cite{Sco65} that this set is actually always Borel and thus any countable structure admits a characterization up to isomorphism. Scott's argument uses the back-and-forth relations\footnote{Technically, Scott's original argument used a different symmetric variant of the back-and-forth relations, but all of the ideas are the same whichever variant one uses.}: He shows that for any structure $\mc{M}$ there is a countable ordinal $\alpha$ such that for any countable structure $\mc{N}$,
	\[\mc{M} \leq_\alpha \mc{N} \Longleftrightarrow \mc{M} \cong \mc{N}.\]
	From this, one checks that the set $\{\mc{N} : \mc{M} \leq_\alpha \mc{N}\}$ is $\bfPi^0_{2\alpha}$ by writing out the definition of the back-and-forth relations---each step of the inductive definition uses two quantifiers---and concludes that $\{ \mc{N} : \mc{M} \cong \mc{N}\}$ is $\bfPi^0_{2\alpha}$ and hence Borel. (The fact that the back-and-forth relations $\leq_\alpha$ have a $\Pi_{2\alpha}$ definition is well-known and is used in various contexts. It appears as Exercise VIII.6 in \cite{MonBook} and underlies, for example, Nadel's \cite{Nadel} observation that a computable structure has a computable Scott sentence if and only if it has computable Scott rank.)
	
	For a given $\mc{M}$, the Borel complexity of this set $\{ \mc{N} : \mc{M} \cong \mc{N}\}$, as measured by Wadge reducibility, is called the \textit{Scott complexity} of $\mc{M}$ \cite{AGNHTT}. By the Lopez-Escobar theorem \cite{Lop66,Vau74}, the Scott complexity of $\{ \mc{N} : \mc{M} \cong \mc{N}\}$ is the same as the least complexity of a Scott sentence for $\mc{M}$, that is, a sentence $\varphi$ of the infinitary logic $\mc{L}_{\omega_1 \omega}$ that characterizes $\mc{M}$ up to isomorphism among countable structures: for any countable structure $\mc{N}$,
	\[ \mc{M} \models \varphi \Longleftrightarrow \mc{N} \cong \mc{M}.\]
	Thus in the context of the mentioned proof of Scott's isomorphism theorem,  $\{ \mc{N} : \mc{M} \cong \mc{N}\}$ is $\bfPi^0_{2\alpha}$ and so is defined by a $\Pi_{2\alpha}$ sentence of $\mc{L}_{\omega_1 \omega}$. This gives an upper bound of $\Pi_{2\alpha}$ on the Scott complexity of $\mc{M}$, but it is far from the lower bound of $\Pi_\alpha$.\footnote{We also note that this is related to the variations between different notions of Scott rank. There have been many different non-equivalent definitions of Scott rank, some of them using the back-and-forth relations, and some using the complexity of Scott sentences.
		See \cite{MonSR} for a survey.}

	In this paper, we ask whether this upper bound can be improved, and if so by how much? We work not just with Scott complexity but also with the problem of defining the back-and-forth relations $\leq_\alpha$ in general. Though we generally state our results in terms of structures, considering sets such as $\{ \mc{N} : \mc{M} \leq_\alpha \mc{N}\}$, by naming constants our results also apply to tuples, e.g., $\{ (\mc{N},\bar{b}) : (\mc{M},\bar{a}) \leq_\alpha (\mc{N},\bar{b})\}$. Such a set represents the $\Pi_\alpha$-type of $\bar{a}$ in $\mc{M}$, and these types are important objects in countable structure theory.
	
	We show that the $\bfPi^0_{2\alpha}$ definition is the best definition of $\leq_\alpha$ in terms of \textit{both} of the structures $\mc{M}$ and $\mc{N}$, that is, the set $\{ (\mc{M},\mc{N}) : \mc{M} \leq_\alpha \mc{N}\}$ is $\bfPi^0_{2\alpha}$-complete for most languages and $\alpha$. Fixing one of the structures, this gives a ``schematic'' $\bfPi^0_{2\alpha}$ definition of the sets $\{\mc{N} : \mc{M} \leq_\alpha \mc{N}\}$ and $\{ \mc{N} : \mc{M} \geq_\alpha \mc{N}\}$ where we replace quantifiers over $\mc{M}$ by infinitary conjunctions and disjunctions. However, this is not the best possible definition of these sets; we show that there are simpler $\bfPi^0_{\alpha+2}$ and $\bfPi^0_{\alpha+3}$ definitions for $\{\mc{N} : \mc{M} \leq_\alpha \mc{N}\}$ and $\{ \mc{N} : \mc{M} \geq_\alpha \mc{N}\}$ respctively. These simpler definitions are not ``schematic'' in $\mc{M}$ and are non-effective. The exact complexities are given by the following theorem and are proved to be best possible. There are exceptional cases when $\alpha$ is too small or too close to a limit ordinal, but for the most part we have the same bounds.
	
	\begin{theorem}\label{thm:main-complexity}
		For any language containing a relation symbol of arity $\geq 2$,\footnote{Any language with a binary relation symbol can effectively bi-interpret any theory in any language. In particular, these maps preserve properties regarding back-and-forth relations, so such languages exhibit all properties exhibited in any language. See \cite{MonBook1} Section VI.3.2 and \cite{MonBook} Theorem XI.7 for more details.} the set
		\[ \{(\mc{M},\mc{N}) : \mc{M} \leq_\alpha \mc{N} \}\]
		is $\bfPi^0_{2\alpha}$-complete. Moreover:
		\begin{enumerate}
			\item The sets
			\[ \{ \mc{N} : \mc{M} \leq_\alpha \mc{N} \} \]
			have the following complexities $\Gamma$ depending on $\alpha$, and there is a structure $\mc{M}$ such that the set is $\Gamma$-complete.
			
			\begin{enumerate}
				\item For $\alpha = 1$, $\Gamma = \bfPi^0_1  =\bfPi^0_\alpha$.
				\item For $\alpha = \lambda$ a limit ordinal, $\Gamma = \bfPi^0_\lambda = \bfPi^0_\alpha$.
				\item For $\alpha = \lambda+1$ the successor of a limit ordinal, $\Gamma = \bfPi^0_{\lambda+2} = \bfPi^0_{\alpha+1}$.
				\item For $\alpha = \beta + 2$, $\Gamma = \bfPi^0_{\alpha+2}$.
			\end{enumerate}
			
			\item The sets
			\[ \{ \mc{N} : \mc{M} \geq	_\alpha \mc{N} \} \]
			have the following complexities $\Gamma$ depending on $\alpha$, and there is a structure $\mc{M}$ such that the set is $\Gamma$-complete.
			
			\begin{enumerate}
				\item For $\alpha = 1$, $\Gamma = \bfPi^0_2 = \bfPi^0_{\alpha+1}$.
				\item For $\alpha = 2$, $\Gamma = \bfPi^0_3 = \bfPi^0_{\alpha+1}$.
				\item For $\alpha = \lambda$ a limit ordinal, $\Gamma = \bfPi^0_\lambda = \bfPi^0_\alpha$.
				\item For $\alpha = \lambda+1$ the successor of a limit ordinal, $\Gamma = \bfPi^0_{\lambda+2} = \bfPi^0_{\alpha+1}$.
				\item For $\alpha = \lambda + 2$ the double successor of a limit ordinal, $\Gamma = \bfPi^0_{\lambda+4} = \bfPi^0_{\alpha+2}$.
				\item For $\alpha = \beta + 3$, $\Gamma = \bfPi^0_{\alpha+3}$.
			\end{enumerate}
		\end{enumerate}
	\end{theorem}
	
	\noindent The first part of the paper is devoted to proving this theorem. We note that while $\{ (\mc{M},\mc{N}) : \mc{M} \leq_\alpha \mc{N}\}$ is (lightface) $\Pi^0_{2\alpha}$ when $\alpha$ is computable, the other bounds where $\mc{M}$ is fixed are not generally effective in $\mc{M}$. That is, for example, if $\mc{M}$ is computable, we do not know whether $\{\mc{N} : \mc{M} \leq_n \mc{N}\}$ is (lightface) $\Pi^0_{n+2}$. In \cite{AlvirCsimaHarrisonTrainor} Alvir, Csima, and Harrison-Trainor undertake an analysis of the computability of Scott sentences for computable structures with $\Pi_2$ Scott sentences; it is shown that a computable structure $\mc{M}$ can have a $\Pi_2$ Scott sentence but no computable $\Sigma_4$ Scott sentence. This implies that for this particular $\mc{M}$ the set $\{\mc{N} : \mc{M} \leq_2 \mc{N}\}$ is (boldface) $\bfPi^0_2$ but not (lightface) $\Sigma^0_4$.
	
	\begin{question}
		Assuming that $\mc{M}$ is computable, what is the lightface complexity of the sets from Theorem \ref{thm:main-complexity}?
	\end{question}
	
	\noindent The fact that $2+2 = 2 \cdot 2$ means that in the case $\alpha = 2$ the complexities $\bfPi^0_{2\alpha}$ and $\bfPi^0_{\alpha+2}$ are the same, and so we get no evidence about the behavior in general.
	
	In the second part of the paper, we will consider the following question. What class of $\mc{L}_{\omega_1 \omega}$ formulas corresponds \textit{exactly} with the back-and-forth relations? Namely, we want a class of formulas $\Gamma_\alpha$ that has the properties:
	\begin{enumerate}
		\item for each $\mc{M}$ and $\alpha$ there is a sentence $\varphi \in \Gamma_\alpha$ such that
		\[   \mc{N} \models \varphi \Longleftrightarrow \mc{M} \leq_\alpha \mc{N},\]
		\item for any $\varphi \in \Gamma_\alpha$, if $\mc{M} \leq_\alpha \mc{N}$, and $\mc{M} \models \varphi$, then $\mc{N} \models \varphi$.
	\end{enumerate}
	While $\Pi_\alpha$ formulas satisfy (2), above we saw that one requires $\Pi_{\alpha+2}$ formulas for (1). We will introduce a new hierarchy of formulas, which we call the $\E_\alpha/\A_\alpha$ hierarchy, which is better suited for characterizing the back-and-forth relations. The definition is somewhat complicated, but the essence is that when counting quantifiers, $\forall \bigdoublevee \bigdoublewedge \exists$ counts only as two alternations of quantifiers, whereas in the $\Sigma_\alpha$/$\Pi_\alpha$ hierarchy, this would count for four alternations of quantifiers. There are also associated $\vwE_\alpha$ and $\vwA_\alpha$ formulas which are the closure, under $\bigdoublewedge$ and $\bigdoublevee$, of the $\E_\alpha$ and $\A_\alpha$ formulas respectively. We leave the full definition to Section \ref{sec:defineAE}. We prove that these classes of formulas satisfy (1) and (2) above:
	\begin{restatable}{proposition}{thmbnfformulas}\label{thm:bnfFormulas}
		For $\mc{M}$ a countable structure, $\bar{a} \in \mc{M}$, and $\alpha \geq 1$, there are a $\vwE_\alpha$ formula $\varphi_{\bar{a},\mc{M},\alpha}(\bar{x})$ and $\A_\alpha$ formula $\psi_{\bar{a},\mc{M},\alpha}(\bar{x})$ such that for $\mc{N}$ any structure,
		\[ \mc{N} \models \varphi_{\bar{a},\mc{M},\alpha}(\bar{b}) \Longleftrightarrow (\mc{M},\bar{a}) \geq_\alpha (\mc{N},\bar{b}) \]
		and
		\[ \mc{N} \models \psi_{\bar{a},\mc{M},\alpha}(\bar{b}) \Longleftrightarrow (\mc{M},\bar{a}) \leq_\alpha (\mc{N},\bar{b}).\]
	\end{restatable}
	
	\begin{restatable}{proposition}{proptransferover}\label{prop:transfer-over-bf}
		Suppose that $(\mc{M},\bar{a}) \leq_\alpha (\mc{N},\bar{b})$ for $\alpha \geq 1$. Then given a $\vwE_\alpha$ formula $\varphi(\bar{x})$ and a $\vwA_\alpha$ formula $\psi(\bar{x})$,
		\[ \mc{N} \models \varphi(\bar{b}) \Longrightarrow \mc{M} \models \varphi(\bar{a})\]
		and
		\[ \mc{M} \models \psi(\bar{a}) \Longrightarrow \mc{N} \models \psi(\bar{b})\]
	\end{restatable}
	At first, we thought that these classes of formulas were only a curiosity, but they have proved very useful. For example, they appear naturally in the proof of the upper bounds in 1(d) and 2(d) of Theorem \ref{thm:main-complexity}. More importantly, they have been useful for certain Henkin constructions. The most important example of this is in \cite{GHT}, where the proof of the main result is a Henkin construction using formulas of complexity $\E_\alpha$. We will discuss briefly in Section \ref{sec:henkin} the role that $\E_\alpha$ formulas play in that construction. Finally, in Section \ref{sec:omittingtypes}, we give a Henkin construction using $\E_\alpha$ formulas to give an improved version of Montalb\'an's type omitting theorem for infinitary logic \cite{MonSR}. This type omitting theorem was originally due to Gonzalez but was unpublished; we give a new and easier proof here making use of the $\A_\alpha/\E_\alpha$ formulas.
	
	\begin{theorem}\label{thmomitting1}
		Let $\mc{M}$ be a countable structure. Let $(\Gamma_i)_{i \in \omega}$ be a list of $\Pi_{\alpha}$-types which are not $\Sigma_{\alpha}$-supported in $\mc{M}$. Let $\varphi$ be a $\Pi_{\alpha+1}$ sentence true in $\mc{M}$. Then there is $\mc{N} \leq_\alpha \mc{M}$ such that $\mc{N} \models \varphi$ and $\mc{N}$ omits all of the $\Gamma_i$.
	\end{theorem}
	
	\noindent This adds the additional condition that $\mc{N} \leq_\alpha \mc{M}$ to Montalb\'an's type omitting theorem. In particular, we maintain the entire $\Sigma_\alpha$ theory of $\mc{M}$. It is still open whether we can have $\mc{N} \geq_\alpha \mc{M}$ or even stronger $\mc{N} \equiv_\alpha \mc{M}$.
	
	In Section \ref{sec:omittingtypes}, we state several consequences of this type-omitting theorem such as the following.
	
	\begin{restatable}{corollary}{cor}
		Let $\mc{M}$ be a countable structure.
		\begin{enumerate}
			\item If, for all countable $\mc{N}$,
			\[ \mc{M} \leq_\alpha \mc{N} \Longrightarrow \mc{M} \cong \mc{N} \]
			then $\mc{M}$ has a $\Pi_{\alpha + 2}$ Scott sentence.
			\item If, for all countable $\mc{N}$,
			\[ \mc{N} \leq_\alpha \mc{M} \Longrightarrow \mc{M} \cong \mc{N} \]
			then $\mc{M}$ has a $\Pi_{\alpha + 1}$ Scott sentence.
		\end{enumerate}
	\end{restatable}
	
	Item (1) simply follows from our general characterization of $\{\mc{N} : \mc{M} \leq_\alpha \mc{N}\}$ as a $\bfPi^0_{\alpha + 2}$ set for any $\alpha$. For (2), such an argument would yield a $\Pi_{\alpha+3}$ Scott sentence. Using the type omitting theorem, we improve this to a $\Pi_{\alpha+1}$ Scott sentence, and this cannot be further improved; for example, take $\mc{M}$ to be the infinite-dimensional $\mathbb{Q}$-vector space. This has a $\Pi_3$ Scott sentence and no simpler, but if $\mc{V} \leq_2 \mathbb{Q}^{\mathbb{N}}$ then $\mc{V} \cong \mathbb{Q}^{\mathbb{N}}$. Whether (1) is optimal remains open.
	
	\begin{restatable}{question}{optimalques}\label{q:sr}
		Let $\mc{M}$ be a countable structure. Suppose that for all countable $\mc{N}$
		\[\mc{M} \leq_\alpha \mc{N} \Longrightarrow \mc{M} \cong \mc{N}.\]
		Must $\mc{M}$ have a $\Pi_{\alpha+1}$ Scott sentence? A $\Pi_{\alpha}$ Scott sentence?
	\end{restatable}
	
	\section{The definability of the back-and-forth relations in the $\Sigma$ and $\Pi$ hierarchies}
	
	We begin by demonstrating the lower bound of $\bfPi^0_{2n}$ in the symmetrical case of the theorem for finite values.
	In other words, we demonstrate that the set $\{(\mc{M},\mc{N}) : \mc{M} \leq_n \mc{M} \}$ is sometimes as difficult to define as possible.
	Note that $\bfPi^0_{2n}$ matches the classical upper bound, as the definition of the back-and-forth game gives a $\bfPi^0_{2n}$ description of the back-and-forth relations.
	(We expand on this point regarding the classical upper bound in the proof below.)

	\begin{lemma}\label{lem:uniform}
		For any language $\mc{L}$ containing a relation symbol of arity $\geq 2$, the set
		\[ \{(\mc{M},\mc{N}) : \mc{M} \leq_n \mc{N} \}\]
		is $\bfPi^0_{2n}$-complete in $\Mod(\mc{L}) \times \Mod(\mc{L})$.
	\end{lemma}
	\begin{proof}
		Any language containing a relation symbol of arity $\geq 2$ can simulate, via effective bi-interpretation, any other language. Thus, in showing completeness we can work with whichever language is most useful.
		
		The fact that this set is $\bfPi^0_{2n}$ has been observed previously (see, e.g., Lemma VI.14 in \cite{MonBook}).
		To be a bit more explicit, the $\bfPi^0_{2n}$ definition follows directly from the definition of the back-and-forth relations.
		In the product space, $\Mod(\mc{L}) \times \Mod(\mc{L})$, a universal quantifier over the elements of the first coordinate (or second coordinate) naturally corresponds to a countable intersection of sets.
		Dually, an existential quantifier naturally corresponds to a countable union.
		The definition of $\leq_n$ alternates quantifiers (some over the first structure, some over the second structure) $2n$-times and begins with a universal quantifier, so it gives a $\bfPi^0_{2n}$ definition.
		
		We focus the rest of the proof on showing that the set is $\bfPi^0_{2n}$-hard.
		We will give an inductive construction but we need a stronger statement to carry out the induction. For $n \geq 2$, given a $\Pi^0_{2n-2}$ set $U \subseteq 2^\omega \times \omega$ and $x \in 2^\omega$ we will construct structures $\mc{A}^{x,U}$ and $\mc{B}^{x,U}_i$ such that
		\begin{enumerate}
			\item if $(x,i) \in U$ then $\mc{A}^{x,U} \geq_n \mc{B}^{x,U}_{i}$, and
			\item if $(x,i) \notin U$ then $\mc{B}^{x,U}_{i} \ngeq_{n-1} \mc{A}^{x,U}$.
		\end{enumerate}
		For a given value of $n$, the language of these structures will consist of equivalence relations $\mathbb{E}_{n},\ldots,\mathbb{E}_3$ and unary relations $R^{n'}_{i,m}$ for $2 \leq n' \leq n$ and $i,m \in \omega$. The equivalence relation $\mathbb{E}_{i}$ will refine $\mathbb{E}_j$ if $i < j$. Inductively, in constructing these structures for $n+1$, we may assume that we know how to construct them for $n$.
		
		Our base case is $n = 2$.
		We assume that we work relative to an oracle that makes a given $\bfPi^0_{2}$ set $\Pi^0_{2}$ and note that our proof relativizes; we make a similar assumption as needed throughout this proof.
		Given a $\Pi^0_{2}$ set $V \subseteq 2^\omega \times \omega$, $x \in 2^\omega$, and $k \in \omega$ we must construct structures $\mc{A}^{x,V}$ and $\mc{B}^{x,V}_i$ with the desired properties. Write
		\[ (x,i) \in V \Longleftrightarrow \forall m \exists j \; (x,i,m,j) \in U,\]
		where $U$ is computable.
		We will have unary relations $R^2_{i,m}$. 
		Let $\mc{A}^{x,V}$ have infinitely many elements $a^s_{i',m,j}$.
		The elements $a^s_{i',m,j}$ will behave the same for fixed values of $i',m,j$.
		In other words, each element is infinitely replicated by the index $s$.
		Put $a^s_{i',m,j} \in R^2_{i',m}$ if $(x,i',m,j) \in U$. Let $\mc{B}^{x,V}_i$ have infinitely many elements $b^s_{i',m,j}$ as well as elements $b_m^*$.
		Put $b^s_{i',m,j} \in R^2_{i',m}$ if $(x,i',m,j) \in U$ and put $b_m^* \in R^2_{i,m}$.
		Informally, $\mc{B}^{x,V}_i$ is just the same as $\mc{A}^{x,V}$ except for the fact that there is an additional element added to $R^2_{i,m}$ for each $m$.
		Then:
		\begin{itemize}
			\item If $(x,i) \in V$, then we can check that $\mc{A}^{x,V} \cong \mc{B}^{x,V}_{i}$ and hence $\mc{A}^{x,V} \geq_2 \mc{B}^{x,V}_{i}$. We see that $\mc{A}^{x,V}$ embeds into  $\mc{B}^{x,V}_{i}$. The elements not in the image are the $b_m^*$. For each $m$ there is $j_m$ such that $(x,i,m,j_m) \in U$ and so $b_m^*$ satisfies the same relations as each $a^s_{i,m,j_m}$ (for each $s$). Thus we can map, say $b_m^* \leftrightarrow a^0_{i,m,j_m}$ and $b^s_{i,m,j} \leftrightarrow a^{s+1}_{i,m,j}$.
			\item If $(x,i) \notin V$ then there is $m$ such that for all $j$ we have $(x,i,m,j) \notin U$. Note that in $\mc{B}^{x,V}_{i}$ there is always an element $b_m^*$ satisfying $R^2_{i,m}$.
			In contrast, in $\mc{A}^{x,V}$, any $a_{i',m',j}$ satisfying $R^2_{i,m}$ must have $i' = i$ and $m' = m$ and $(x,i,m,j) \in U$.
			This cannot happen by the choice of $i$ and $m$, and so we have an existential formula true in $\mc{B}^{x,V}_{i}$ which is not true in $\mc{A}^{x,V}$. Thus $\mc{B}^{x,V}_{i} \ngeq_1 \mc{A}^{x,V}$.
		\end{itemize}
		This finishes the case of $n = 2$.
		
		Now, suppose that we have a construction for $n$, and we will give a construction for $n+1$. Let $V$ be a $\Pi^0_{2n}$ set, say
		\[ x \in V \Longleftrightarrow \forall m \exists j \; (x,i,m,j) \in U \]
		where $U$ is $\Pi^0_{2n-2}$. For each $i,m$, let
		\[ U_{i,m} = \{ (x,j) \; : \; (x,i,m,j) \in U\}.\]
		By the inductive hypothesis, we have structures $\mc{A}^{x,U_{i,m}}$ and $\mc{B}^{x,U_{i,m}}_j$ such that
		\begin{enumerate}
			\item  if $(x,j) \in U_{i,m}$ then $\mc{A}^{x,U_{i,m}} \geq_{n} \mc{B}^{x,U_{i,m}}_{j}$, and
			\item if $(x,j) \notin U_{i,m}$ then  $\mc{B}^{x,U_{i,m}}_{j} \ngeq_{n-1} \mc{A}^{x,U_{i,m}}$.
		\end{enumerate}
		To be explicit, the above definitions and hypotheses immediately give that 
		\[(x,i)\in V \iff \forall m\exists j ~(x,j)\in U_{i,m}.\]
		Therefore,
		\[(x,i)\in V\implies \forall m\exists j ~ \mc{A}^{x,U_{i,m}} \geq_{n} \mc{B}^{x,U_{i,m}}_{j} \text{  and  } (x,i)\not\in V\implies \exists m\forall j ~ \mc{A}^{x,U_{i,m}} \nleq_{n-1} \mc{B}^{x,U_{i,m}}_{j}.\]
		
		Define $\mc{A}^{x,V}$ and $\mc{B}^{x,V}_i$ as follows:
		\begin{itemize}
			\item $\mc{A}^{x,V}$ is split into sorts indexed by $i,m$ by the relations $R^{n+1}_{i,m}$.
			Each sort is, in turn, split into equivalence classes with respect to the equivalence relation $\mathbb{E}_{n+1}$ with each equivalence class being the domain of one of the structures from the previous inductive case.
			The $(i,m)$ sort has infinitely many equivalence classes: For each  $j\in\omega$ there are infinitely many classes containing a copy of $\mc{B}^{x,U_{i,m}}_j$ for each $j\in\omega$, and there are no other equivalence classes.
			\item $\mc{B}^{x,V}_i$ is similarly split into sorts indexed by $i,m$ by the relations $R^{n+1}_{i,m}$.
			Again, each sort is split into equivalence classes with respect to the equivalence relation $\mathbb{E}_{n+1}$, and each equivalence class is the domain of a structure.
			The $(i',m)$ sort has infinitely many equivalence classes of the following types: For each $j \in \omega$ there are infinitely many equivalence classes containing a copy of $\mc{B}^{x,U_{i',m}}_j$, and if $i'=i$, there are also infinitely many equivalence classes containing a copy of $\mc{A}^{x,U_{i,m}}$; there are no other equivalence classes.
		\end{itemize}
		Informally, $\mc{B}^{x,V}_i$ is just the same as $\mc{A}^{x,V}$ except for the fact that there are infinitely many additional copies of $\mc{A}^{x,U_{i,m}}$ added to $R^{n+1}_{i,m}$ for each $m$.
		We now demonstrate that the $\mc{A}^{x,V}$ and $\mc{B}^{x,V}_i$ have Properties (1) and (2) required to continue the induction.
		
		\begin{sublemma}\label{sublemma:complicated}
			Let $\mc{A}$ and $\mc{B}$ be structures in the relational language $\mc{L} \cup \{\mathbb{E}\}$ which consist of an equivalence relation $\mathbb{E}$ with infinitely many equivalence classes, such that each equivalence class is an $\mc{L}$-structure (and no relations from $\mc{L}$ hold between different equivalence classes). Write $\mc{A} = \bigsqcup_i \mc{A}_i$ and $\mc{B} = \bigsqcup_i \mc{B}_i$ as disjoint unions of equivalence classes, with the $\mc{A}_i$ and $\mc{B}_i$ being $\mc{L}$-structures. Suppose furthermore that each isomorphism type that appears does so infinitely many times. If the following conditions hold, then $(\mc{A},\bar{a}) \geq_{n} (\mc{B},\bar{b})$:
			\begin{enumerate}[label=(\Alph*)]
				\item For each $k,\ell$, $a_k \mathbb{E} a_\ell$ if and only if $b_k \mathbb{E} b_\ell$.
				\item Given \textup{(A)}, we may reorder $\bar{a}$ and $\bar{b}$ into subtuples $\bar{a}_1,\ldots,\bar{a}_\ell$ and $\bar{b}_1,\ldots,\bar{b}_\ell$ such that there are $i_1,\ldots,i_\ell$ and $j_1,\ldots,j_\ell$ with $\bar{a}_k \in \mc{A}_{i_k}$ and $\bar{b}_k \in \mc{B}_{j_k}$. Then for $k = 1,\ldots,\ell$,
				\[ (\mc{A}_{i_k},\bar{a}_k) \geq_n (\mc{B}_{i_k},\bar{b}_k).\]
				\item For each $i$ there is $j$ such that $\mc{A}_i \geq_n \mc{B}_j$.
				\item For each $j$ there is $i$ such that $\mc{B}_j \geq_{n-1} \mc{A}_i$.
			\end{enumerate}
		\end{sublemma}
		\begin{proof}
			We argue inductively on $n$, starting with the base case of $n = 0$. This case is immediate. 
			Supposing that (A), (B), and (C) hold, we see that these determine for each type of atomic formula that it is true of $\bar{a}$ in $\mc{A}$ if and only if it is true of $\bar{b}$ in $\mc{B}$.
			
			Now we inductively assume we know the sublemma for $n$, and prove it for $n+1$. We argue that if (A), (B), (C), and (D) hold for $n+1$, then $(\mc{A},\bar{a}) \geq_{n+1} (\mc{B},\bar{b})$. Suppose that we are given $\bar{a}' \in \mc{A}$. First, since (A) holds, we may break up $\bar{a}$ and $\bar{b}$ as in (B) into tuples $\bar{a}_1,\ldots,\bar{a}_\ell$ and $\bar{b}_1,\ldots,\bar{b}_\ell$. We may break $\bar{a}'$ up into tuples $\bar{a}_1',\ldots,\bar{a}_\ell'$ such that $\bar{a}_k' \in \mc{A}_{i_k}$, and further elements $\bar{a}_1,\ldots,\bar{a}_s^*$ such that $\bar{a}_t^* \in \mc{A}_{i_t^*}$. By (B), for each $k = 1,\ldots,\ell$ we have $(\mc{A}_{i_k},\bar{a}_k) \geq_{n+1} (\mc{B}_{i_k},\bar{b}_k)$ and so we may pick $\bar{b}_k' \in \mc{B}_{i_k}$ such that $(\mc{B}_{i_k},\bar{b}_k \bar{b}_k')  \geq_{n}  (\mc{A}_{i_k},\bar{a}_k \bar{a}_k')$. By (C), for each $t = 1,\ldots,s$ we may choose $j_t^*$ and $\bar{b} \in \mc{B}_{j_t^*}$ such that $(\mc{B}_{j_t^*},\bar{b}_t^*) \geq_{n} (\mc{A}_{i_t^*},\bar{a}_t^*)$. Furthermore, as each isomorphism type appears infinitely many times, we may choose each of the $j_t^*$ to be an index that we have not yet chosen. Our goal is to argue that 
			\[ (\mc{B},\bar{b}_1\bar{b}_1',\ldots,\bar{b}_k\bar{b}_k',\bar{b}_1^*,\ldots,\bar{b}_s^*) \geq_n (\mc{A},\bar{a}_1\bar{a}_1',\ldots,\bar{a}_k\bar{a}_k',\bar{a}_1^*,\ldots,\bar{a}_s^*)\]
			from which we can conclude that
			\[ (\mc{A},\bar{a}) = (\mc{A},\bar{a}_1,\ldots,\bar{a}_k) \geq_{n+1} (\mc{B},\bar{b}_1,\ldots,\bar{b}_k) = (\mc{B},\bar{b}).\]
			To show the former, we will use the inductive hypothesis. (A) and (B) follow by choice of the $\bar{b}_i'$ and $\bar{b}_i^*$. (C) for $n$ is exactly (D) for $n+1$, and (D) follows from (C) as $\mc{A}_i \geq_{n+1} \mc{B}_j$ implies $\mc{A}_i \leq_{n} \mc{B}_j$.
		\end{proof}
		
		\begin{claim}\label{claim:(1)}
			If $(x,i) \in V$ then $\mc{A}^{x,V} \geq_{n+1} \mc{B}^{x,V}_i$.
		\end{claim}
		\begin{proof}
			$\mc{A}^{x,V}$ and $\mc{B}^{x,V}_i$ are of the form specified by the previous claim.
			Furthermore, each sort indexed by $i',m$ is of the specified form, so we may check that the conditions of the previous claim are satisfied sort by sort.
			If $i'\neq i$, the $(i',m)$\textsuperscript{th} sort of $\mc{A}^{x,V}$ and $\mc{B}^{x,V}_i$ are isomorphic by construction, so there is nothing non-trivial to check.
			Consider instead when $i'=i$.
			(A) and (B) are vacuously true since there are no tuples.
			(C) is due to the fact that each structure $\mc{B}_j^{x,U_{i,m}}$ appearing as an equivalence class in $\mc{A}^{x,V}$ also appears as an equivalence class in $\mc{B}_i^{x,V}$. 			(D) is verified as follows.
			Within the $(i,m)$\textsuperscript{th} sort, the structures appearing as equivalence classes in $\mc{B}_i^{x,V}$ are the $\mc{B}_j^{x,U_{i,m}}$, which are also equivalence classes of $\mc{A}^{x,V}$, and also the $\mc{A}^{x,U_{i,m}}$.
			For each $m$, since $(x,i) \in V$, there is $j$ such that $(x,i,m,j) \in U$, so that $(x,j) \in U_{i,m}$.
			Thus, for each $m$ there is $j$ such that $\mc{A}^{x,U_{i,m}} \geq_{n} \mc{B}^{x,U_{i,m}}_{j}$.
			Since we have checked (A), (B), (C), and (D) of the previous claim, we conclude that $\mc{A}^{x,V} \geq_{n+1} \mc{B}^{x, V}_i$.
		\end{proof}
		
		\begin{sublemma}\label{sublemma:complicated2}
			Let $\mc{A}$ and $\mc{B}$ be structures in the relational language $\mc{L} \cup \{\mathbb{E}\}$ which consist of an equivalence relation $\mathbb{E}$ with infinitely many equivalence classes, such that each equivalence class is an $\mc{L}$-structure (and no relations from $\mc{L}$ hold between different equivalence classes). Write $\mc{A} = \bigsqcup_i \mc{A}_i$ and $\mc{B} = \bigsqcup_i \mc{B}_i$ as disjoint unions of equivalence classes, with the $\mc{A}_i$ and $\mc{B}_i$ being $\mc{L}$-structures. If there is some $\mc{B}_j$ such that for every $\mc{A}_i$ we have $\mc{A}_i \ngeq_{n-1} \mc{B}_j $, then $\mc{B} \ngeq_{n} \mc{A}$. 
		\end{sublemma}
		\begin{proof}
			Suppose to the contrary that  $\mc{B} \geq_{n} \mc{A}$. Choose $\bar{b} \in \mc{B}_j$. There is $\bar{a} \in \mc{A}$ such that $(\mc{A},\bar{a}) \geq_{n-1} (\mc{B},\bar{b})$. But, as $\mc{A} \ngeq_{n-1} \mc{B}$, this yields a contradiction.
		\end{proof}
		
		\begin{claim}\label{claim:(2)}
			If $(x,i) \notin V$ then $\mc{B}^{x,V}_i \ngeq_n \mc{A}^{x,V}$.
		\end{claim}
		\begin{proof}
			Since $(x,i) \notin V$, there is $m$ such that for all $j$ we have $(x,i,m,j) \notin U$ and so $(x,j) \notin U_{i,m}$. Thus, there is $m$ such that for all $j$ we have $\mc{B}^{x,U_{i,m}}_{j} \ngeq_{n-1} \mc{A}^{x,U_{i,m}}$. 
			For this $m$, we will show that Sublemma \ref{sublemma:complicated2} applies when focusing on the $(i,m)$\textsuperscript{th} sort of $\mc{B}^{x,V}_i$ and $ \mc{A}^{x,V}$
			
			$\mc{A}^{x,U_{i,m}}$ is one of the equivalence classes of $\mc{B}^{x,V}_i$ for this sort.
			The equivalence classes of $\mc{A}^{x,V}$ for this sort are exactly the structures $\mc{B}^{x,U_{i,m}}_j$. 
			Note that for all $j$ we have $(x,j) \notin U_{i,m}$ and so by (2),  $\mc{B}^{x,U_{i,m}}_j \ngeq_{n-1} \mc{A}^{x,U_{i,m}}$.
			Thus, we can apply Sublemma \ref{sublemma:complicated2} to conclude that $\mc{B}^{x,V}_i \ngeq_n \mc{A}^{x,V}$ on the sort indexed by $i,m$ and indeed $\mc{B}^{x,V}_i \ngeq_n \mc{A}^{x,V}$ follows at once.
		\end{proof}

		Claims \ref{claim:(1)}, and \ref{claim:(2)} finish the inductive argument.
		We now argue for the full statement of the lemma.
		Consider a $\bfPi^0_{2m}$ set $V\subseteq 2^\omega$.
		Note that $U:=V\times \{0\}$ is a $\Pi_{2m}$ subset of $2^\omega\times\omega$.
		The above inductive construction (with parameter $n=m+1$) gives us structures $\mc{M}^x = \mc{A}^{x,U}$ and $\mc{N}^x = \mc{B}^{x,U}_0$  such that
		\[x\in V \iff (x,0)\in U \iff \mc{A}^{x,U} \leq_m \mc{B}^{x,U}_0  \iff \mc{M}^x \leq_m \mc{N}^x.\]
		This gives the desired effective reduction from $V$ to the sets of pairs of structures $(\mc{M},\mc{N})$ related by $\mc{M} \leq_m \mc{N}$ and completes the proof of the lemma.
	\end{proof}

	We delay giving a lower bound on the definability of the back-and-forth relations in the infinite case until Section \ref{sec:jump-inv} where we use jump inversions.
	The finite cases contain the critical combinatorics.
	
	\section{Upper bounds for the definability of the $\alpha$-types of structures}
	
	The remaining statements of Theorem \ref{thm:main-complexity} are concerned with the case where we fix a structure $\mc{M}$ and consider whether some other structure $\mc{N}$ satisfies $\mc{N} \geq_\alpha \mc{M}$ or $\mc{N} \leq_\alpha \mc{M}$. 
	In this section, we will prove the needed upper bounds on the complexities of the sets $\{ \mc{N} : \mc{M} \leq_\alpha \mc{N}\}$ and $\{ \mc{N} : \mc{M} \geq_\alpha \mc{N}\}$. We begin with general bounds which hold for most $\alpha$, followed by certain exceptional cases where we can prove better bounds for small $\alpha$ or $\alpha$ near limit ordinals.
	All of these results are optimal; we provide the corresponding proofs of the lower bounds in future sections.
	
	\subsection{General bounds}
	
	The following two lemmas provide general bounds for the $\alpha$-types of structures that are superior to the $\mathbf{\Pi}^0_{2\alpha}$ bounds provided by the previous analysis.

	\begin{lemma}\label{lem:aboveGeneral}
		Let $\mc{M}$ be a fixed countable structure and let $\bar{a} \in \mc{M}$. For each $\alpha < \omega_1$ there is a $\Pi_{\alpha+2}$ formula $\varphi$
		such that $\mc{N} \models \varphi(\bar{b})$ if and only if $(\mc{N},\bar{b}) \geq_\alpha (\mc{M},\bar{a})$. Hence the set
		\[ \{\mc{N} : \mc{M} \leq_\alpha \mc{N}\}\]
		is $\bfPi^0_{\alpha+2}$.
	\end{lemma}
	
	\begin{proof}
		The classical analysis of the back-and-forth relations (see, e.g., Lemma VI.14 in \cite{MonBook}) gives several base cases for this claim.
		The formulas given there are $\Pi_{2\alpha}$, and for $\alpha=1$ or $2$ and $\alpha=\lambda$ or $\lambda+1$ for a limit level $\lambda$ we have $2\alpha\leq\alpha+2$ and so the claim is already shown.
		(In fact, for $\alpha = 1$, there is even a $\Pi_1$ formula satisfying the lemma.)
		
		We now complete the proof by proving the following inductive claim.
		\begin{claimstar}
			If for every $\bar{a}\in\mc{M}$ there is a formula $\varphi^\alpha_{\mc{M},\bar{a}}(\bar{x})$ satisfying the claim for $\alpha\in\omega_1$ then for every $\bar{a}\in\mc{M}$ there is a formula $\varphi^{\alpha+2}_{\mc{M},\bar{a}}(\bar{x})$ satisfying the claim for $\alpha+2$.
		\end{claimstar}
		
		\begin{proof}	
			Suppose that for each $\mc{M}$ and $\bar{q} \in \mc{M}$ there is a $\Pi_{\alpha+2}$ formula $\varphi^\alpha_{\mc{M},\bar{q}}(\bar{x})$
			such that $\mc{N} \models \varphi(\bar{p})$ if and only if $(\mc{N},\bar{p}) \geq_\alpha (\mc{M},\bar{q})$. 
			Given $\mc{M}$ and $\bar{a} \in \mc{M}$ we must show that there is a a $\Pi_{\alpha+4}$ formula $\varphi(\bar{x})$
			such that $\mc{N} \models \varphi(\bar{p})$ if and only if $(\mc{N},\bar{p}) \geq_{\alpha+2} (\mc{M},\bar{a})$. 
			For each $\bar{b} \in \mc{M}$, there is a $\Pi_{\alpha}$ formula $\theta_{\bar{b}}(\bar{x},\bar{y})$ defining the $\Pi_\alpha$ type of $\bar{a}\bar{b}$ in $\mc{M}$, i.e., such that $\mc{M} \models \theta_{\bar{b}}(\bar{a},\bar{c})$ if and only if $(\mc{M},\bar{a}\bar{c}) \geq_\alpha (\mc{M},\bar{a}\bar{b})$ (see \cite{MonBook} Lemma II.62). Let
			\[ \psi(\bar{x}) := \bigdoublewedge_\ell \forall y_1,\ldots,y_\ell \bigdoublewedge_{b_1,\ldots,b_\ell \in \mc{M}} \left( \theta_{\bar{b}}(\bar{x},\bar{y}) \longrightarrow \varphi^\alpha_{\mc{M},\bar{a}\bar{b}}(\bar{x},\bar{y}) \right)\]
			and let
			\[ \chi(\bar{x}) := \bigdoublewedge_k \forall y_1,\ldots,y_k \bigdoublevee_{b_1,\ldots,b_k \in \mc{M}} \bigdoublewedge_{c_1,\ldots,c_\ell \in \mc{M}} \exists z_1,\ldots,z_\ell \; \theta_{\bar{b}\bar{c}}(\bar{x},\bar{y},\bar{z}) .\]
			Note that $\psi$ is $\Pi_{\alpha+2}$ and $\chi$ is $\Pi_{\alpha+4}$, so that $\psi \wedge \chi$ is $\Pi_{\alpha+4}$. We claim that $\mc{N} \models \psi(\bar{p}) \wedge \chi(\bar{p})$ if and only if $(\mc{N},\bar{p}) \geq_{\alpha+2} (\mc{M},\bar{a})$ and therefore $\psi \wedge \chi$ is our desired formula.
			
			First, suppose that $\mc{N} \models \psi(\bar{p}) \wedge \chi(\bar{p})$. Given $\bar{q} \in \mc{N}$, by $\mc{N} \models \chi(\bar{p})$ there is $\bar{b} \in \mc{M}$ such that 
			\[ \mc{N} \models \bigdoublewedge_{c_1,\ldots,c_\ell \in \mc{M}} \exists z_1,\ldots,z_\ell \; \theta_{\bar{b}\bar{c}}(\bar{p},\bar{q},\bar{z}) .\]
			Then for any $\bar{c} \in \mc{M}$, there is $\bar{r} \in \mc{N}$ such that
			\[ \mc{N} \models \theta_{\bar{b}\bar{c}}(\bar{p},\bar{q},\bar{r}).\]
			By $\mc{N} \models \psi(\bar{p})$, we get
			\[ \mc{N} \models \varphi_{\mc{M},\bar{a}\bar{b}\bar{c}}^\alpha(\bar{p},\bar{q},\bar{r}).\]
			This implies that $(\mc{N},\bar{p}\bar{q}\bar{r}) \geq_\alpha (\mc{M},\bar{a}\bar{b}\bar{c})$.
			In particular, to win the $(\mc{N},\bar{p}) \geq_{\alpha+2} (\mc{M},\bar{a})$ game, the Duplicator can play $\bar{b}$ in response to $\bar{q}$ on the first move, $\bar{r}$ in response to $\bar{c}$ on the second move, and then play the winning strategy for $(\mc{N},\bar{p}\bar{q}\bar{r}) \geq_\alpha (\mc{M},\bar{a}\bar{b}\bar{c})$ to win the game, as desired.

			On the other hand, suppose that $(\mc{N},\bar{p}) \geq_{\alpha+2} (\mc{M},\bar{a})$. We must argue that $\mc{N} \models \psi(\bar{p})$ and $\mc{N} \models \chi(\bar{p})$. First, we argue that $\mc{M} \models \psi(\bar{a})$ and $\mc{M} \models \chi(\bar{a})$. The former is due to the definitions of $\theta_{\bar{b}}$ and $\varphi^\alpha_{\mc{M},\bar{a}\bar{b}}$. The latter is because $\mc{M} \models \theta_{\bar{b}\bar{c}}(\bar{a},\bar{b},\bar{c})$ for any $\bar{b}$ and $\bar{c}$ as the quantifiers and conjunctions/disjunctions are essentially both over elements of $\mc{M}$. Now since $\psi$ is $\Pi_{\alpha+2}$, and $(\mc{N},\bar{p}) \geq_{\alpha+2} (\mc{M},\bar{a})$, we have $\mc{N} \models \psi(\bar{p})$. Though $\chi$ is $\Pi_{\alpha+4}$, we can also argue that because $(\mc{M},\bar{a}) \models \chi$ and $(\mc{N},\bar{p}) \geq_{\alpha+2} (\mc{M},\bar{a})$ we get $\mc{N} \models \chi(\bar{p})$. We can argue this by hand, but one could also note that $\chi$ is $\A_{\alpha+2}$ and apply Proposition \ref{prop:transfer-over-bf} (and indeed the complexity of the argument below is one justification for considering this complexity class of formulas as we do in Section \ref{sec:defineAE}). To show that $\mc{N} \models \chi(\bar{p})$, for all $k$ and $\bar{q}$ we must show that
			\[ \mc{N} \models \bigdoublevee_{b_1,\ldots,b_k \in \mc{M}} \bigdoublewedge_{c_1,\ldots,c_\ell \in \mc{M}} \exists z_1,\ldots,z_\ell \; \theta_{\bar{b}\bar{c}}(\bar{p},\bar{q},\bar{z}).\]
			Choose $\bar{u} \in \mc{M}$ such that $(\mc{N},\bar{p}\bar{q}) \leq_{\alpha+1} (\mc{M},\bar{a}\bar{u})$. Since $\mc{M} \models \chi$, we have
			\[\mc{M} \models \bigdoublevee_{b_1,\ldots,b_k \in \mc{M}} \bigdoublewedge_{c_1,\ldots,c_\ell \in \mc{M}} \exists z_1,\ldots,z_\ell \; \theta_{\bar{b}\bar{c}}(\bar{a},\bar{u}
			,\bar{z}).\]
			So, for some disjunct $\bar{b}$ (in fact for $\bar{b} = \bar{u}$) for every $\bar{c}$ there are $\bar{v}_{\bar{c}} \in \mc{M}$ such that
			\[\mc{M} \models \theta_{\bar{b}\bar{c}}(\bar{a},\bar{u}
			,\bar{v}_{\bar{c}}).\]
			For each $\bar{c}$ choose $\bar{r}_{\bar{c}}$ such that $(\mc{M},\bar{a}\bar{u}\bar{v}_{\bar{c}}) \leq_\alpha (\mc{N},\bar{p}\bar{q}\bar{r}_{\bar{c}})$. Since $\theta_{\bar{b}\bar{c}}$ is $\Pi_\alpha$ we get
			\[\mc{N} \models \theta_{\bar{b}\bar{c}}(\bar{p},\bar{q}
			,\bar{r}_{\bar{c}}).\]
			Thus there is $\bar{b}$ such that for every $\bar{c}$ there is $\bar{r}_{\bar{c}}$ such that
			\[ \mc{N} \models \bigdoublevee_{b_1,\ldots,b_k \in \mc{M}} \bigdoublewedge_{c_1,\ldots,c_\ell \in \mc{M}} \exists z_1,\ldots,z_\ell \; \theta_{\bar{b}\bar{c}}(\bar{p},\bar{q},\bar{z}).\]
			Since we have shown this for all $k$ and $\bar{q}$, we have that $\mc{N} \models \chi(\bar{p})$.
		\end{proof}

		This completes the proof of Lemma \ref{lem:aboveGeneral}.
	\end{proof}
	
	Note that for $\alpha$ odd, finite, and greater than $1$ we do not get a better complexity even though the base case is stronger than required by the lemma.
	This is of no help in the inductive steps---the complexity comes from the formula $\chi$, which does not involve the formula given by the inductive hypothesis.
	
	We now prove the analog of Lemma \ref{lem:aboveGeneral} for $\{ \mc{N} : \mc{N} \leq_\alpha \mc{M} \}$.
	
	\begin{lemma}\label{lem:belowGeneral}
		Let $\mc{M}$ be a fixed countable structure. For each $\alpha\in\omega_1$ there is a $\Pi_{\alpha+3}$ formula $\varphi$
		such $\mc{N} \models \varphi(\bar{b})$ if and only if $(\mc{N},\bar{b}) \leq_\alpha (\mc{M},\bar{a})$. Hence the set
		\[ \{\mc{N} : \mc{N} \leq_\alpha \mc{M}\}\]
		is $\bfPi^0_{\alpha+3}$.
	\end{lemma}
	\begin{proof}
		If $\alpha$ is a limit ordinal, this claim follows at once from the classical description of back-and-forth relations.
		Therefore, we can assume that $\alpha=\beta+1$.
		We note that
		\[ (\mc{N},\bar{b}) \leq_\alpha (\mc{M},\bar{a}) \iff \mc{N}\models \bigwwedge_{\beta < \alpha} \bigwwedge_{\bar{c}\in\mc{M}}\exists \bar{x} (\mc{N},\bar{b},\bar{x}) \geq_\beta (\mc{M},\bar{a},\bar{c}).\]
		By Lemma \ref{lem:aboveGeneral} each formula $(\mc{N},\bar{b},\bar{d}) \geq_\beta (\mc{M},\bar{a},\bar{c})$ can be expressed in a $\Pi_{\beta+2}$ manner.
		This means that, overall, the formula is  $\Pi_{\alpha+3}$.    
	\end{proof}
	
	One contrast between the formulas described in Lemmas  \ref{lem:aboveGeneral} and \ref{lem:belowGeneral} to capture the back-and-forth relations and those described in the classical analysis of Lemma VI.14 in \cite{MonBook} lies in the computability of the formulas.
	The method in Lemma VI.14 in \cite{MonBook} is always computable in the simplest presentation of the structure $\mc{M}$.
	On the other hand, Lemmas \ref{lem:aboveGeneral} and \ref{lem:belowGeneral} rely on $\Pi_\alpha$ formulas that isolate the entire $\Pi_\alpha$-type of a tuple \textit{within} a given structure (provided by \cite{MonBook} Lemma II.62).
	Extracting such a formula given a tuple and a structure is not an effective procedure.
	In sum, there is a tradeoff at play here.
	The formulas in Lemmas \ref{lem:aboveGeneral} and \ref{lem:belowGeneral} may be far simpler in quantifier complexity when compared to the previously known descriptions of back-and-forth relations, but they are less computable.
	We conjecture that this tradeoff is necessary. 
	
	\begin{conjecture}
		{\ }
		\begin{enumerate}
			\item For $n \geq 2$ even, there is a structure $\mc{M}$ such that
			\[ \{ \mc{N} : \mc{N} \geq_n \mc{M} \} \]
			is $\Pi^0_{2n}$ but not $\Sigma^0_{2n}$.
			\item For $n \geq 3$ odd, there is a structure $\mc{M}$ such that
			\[ \{ \mc{N} : \mc{N} \geq_n \mc{M} \} \]
			is $\Pi^0_{2n-1}$ but not $\Pi^0_{2n-1}$.
			\item For $n \geq 2$ even, there is a structure $\mc{M}$ such that
			\[ \{ \mc{N} : \mc{N} \leq_n \mc{M} \} \]
			is $\Pi^0_{2n-1}$ but not $\Sigma^0_{2n-1}$.
			\item For $n \geq 3$ odd, there is a structure $\mc{M}$ such that
			\[ \{ \mc{N} : \mc{N} \leq_n \mc{M} \} \]
			is $\Pi^0_{2n}$ but not $\Sigma^0_{2n}$.
		\end{enumerate}
		Moreover, each of these is witnessed by an index set result for countable structures, e.g., for (1), the index set $\{ i : \mc{N}_i \leq_n \mc{M} \}$ is $\Pi^0_{2n}$ $m$-complete.
	\end{conjecture}
	
	Note that the difference between the odd and even cases is due to the base case, and whether the last case of the back-and-forth game is in $\mc{M}$ or $\mc{N}$. See the difference between Lemma \ref{lem:easy1} and Lemma \ref{lem:easy2}.
	
	(1), for example, would be implied by the following conjecture from Scott analysis:
	
	\begin{conjecture}
		For each even $n \geq 2$, there is a computable structure $\mc{M}$ with a $\Pi_n$ Scott sentence but no computable $\Sigma_{2n}$ Scott sentence.
	\end{conjecture}
	
	Alvir, Knight, and McCoy \cite{AlvirKnightMcCoy} produce a structure with a $\Pi_2$ Scott sentence but no computable $\Pi_2$ Scott sentence. Alvir, Csima, and Harrison-Trainor \cite{AlvirCsimaHarrisonTrainor} produce a structure with a $\Pi_2$ Scott sentence but no computable $\Sigma_4$ Scott sentence. The general case remains open.
	
	\subsection{Exceptional bounds}

	We now consider the simplest cases of Theorem \ref{thm:main-complexity} with $\alpha=1$, $\alpha=2$, or $\alpha$ near a limit ordinal.
	These small values have improved descriptions for the $\alpha$-type when compared to the general results of the previous section.
	
	\begin{lemma}\label{lem:easy1}
		For any structure $\mc{M}$ the set
		\[ \{ \mc{N} : \mc{N} \geq_1 \mc{M} \} \]
		is $\bfPi^0_{1}$. 
	\end{lemma}
	\begin{proof}
		The conjunction of the finitary universal sentences true in $\mc{M}$ defines the set as all $\Pi_1$ formulas are a conjunction of these formulas.
	\end{proof}
	
	\begin{lemma}\label{lem:easy2}
		For any structure $\mc{M}$ the set
		\[ \{\mc{N} : \mc{N} \leq_1 \mc{M}\}\]
		is $\bfPi^0_2$.
	\end{lemma}
	\begin{proof}
		This is the same upper bound given previously, straight from the definition. Let $\theta_{\bar{a}}(\bar{x})$ be the finitary quantifier-free formula which says that $\bar{x}$ and $\bar{a}$ satisfy the same atomic formulas (from among the first $|\bar{a}|$-many formulas). Then
		\[ \mc{N}
		\le_1 \@M  \iff  \@N
		\models \bigdoublewedge_{\bar{a} \in \mc{M}^n} \exists y_1,\ldots,y_n \; \theta_{\bar{a}}(y_1,\ldots,y_n).\qedhere\]
	\end{proof}

	\begin{lemma}
		For any structure $\mc{M}$ the set
		\[ \{\mc{N} : \mc{N} \leq_2 \mc{M}\}\]
		is $\mathbf{\Pi}^0_3$.
	\end{lemma}
	
	\begin{proof}
		Note that 
		\[\mc{N} \leq_2 \mc{M}\iff\bigwwedge_{\bar{a}\in \mc{M}}\exists \bar{x}\; (\mc{N},\bar{x})\geq_1(\mc{M},\bar{a})\]
		which is $\bfPi^0_3$ by the upper bound established in Lemma \ref{lem:easy1}.
	\end{proof}
	
	Another place where the bounds are exceptional is near limit ordinals.
	Just like the small finite values, we can improve the bounds for these ordinals.
	We consider these cases below.
	
	\begin{lemma}\label{lem:succlimitupper}
		For any structure $\@M$ we have the following.
		\begin{enumerate}
			\item If $\alpha=\lambda$ where $\lambda$ is a limit ordinal, the set $\{\mc{N} \; : \; \mc{N} \leq_\alpha \mc{M}\}=\{\mc{N} \; : \; \mc{N} \geq_\alpha \mc{M}\}$ is $\bfPi^0_{\alpha}$.
			\item If $\alpha=\lambda+1$ where $\lambda$ is a limit ordinal, the set $\{\mc{N} \; : \; \mc{N} \geq_\alpha \mc{M}\}$ is $\bfPi^0_{\alpha+1}$.
			\item If $\alpha=\lambda+1$ where $\lambda$ is a limit ordinal, the set $\{\mc{N} \; : \; \mc{N} \leq_\alpha \mc{M}\}$ is $\bfPi^0_{\alpha+1}$.
		\end{enumerate}
	\end{lemma}
	
	\begin{proof}
		Each of these follows immediately from the classical description of the back-and-forth relations from Lemma VI.14 in \cite{MonBook}.
		In particular, note that $\lambda=2\lambda$ and $(\lambda+1)+1=2(\lambda+1)$.
	\end{proof}

	\begin{lemma}
		If $\lambda$ is a limit ordinal the set
		\[ \{ \mc{N} : \mc{M} \geq_{\lambda+2} \mc{N} \} \]
		is $\bfPi^0_{\lambda+3}$ for all $\mc{M}$.
	\end{lemma}
	
	\begin{proof}
		For this proof, we let $p_{\bar{a}}$ denote the $\Pi_\lambda$-type of $\bar{a}\in \mc{M}$, and note that $p_{\bar{a}}(\bar{x})$ can be expressed as a $\Pi_\lambda$ formula (as the conjunction of the formulas $\varphi^\alpha_{\bar{a}}(\bar{x})$ for $\alpha < \lambda$ expressing that $(\mc{N},\bar{x}) \geq_\alpha (\mc{M},\bar{a})$). First, if $\mc{M} \geq_{\lambda+2} \mc{N}$, then the $\Pi_{\lambda}$-types realised in $\mc{N}$ are the same as the $\Pi_\lambda$-types realised in $\mc{M}$. (Or, what is the same, $\mc{M} \equiv_{\lambda+1} \mc{N}$). This can be expressed by the $\Pi_{\lambda+2}$ sentence
		\[ \psi = \bigdoublewedge_{\bar{a} \in \mc{M}} \exists \bar{x} \; p_{\bar{a}}(\bar{x}) \;\; \wedge \;\; \bigdoublewedge_n \forall x_1,\ldots,x_n \bigdoublevee_{\bar{a} \in \mc{M}} p_{\bar{a}}(\bar{x}). \]
		Now for each $\bar{a} \in \mc{M}$, let $P_{\bar{a}}$ be the set of all $\Pi_{\lambda}$-types extending that of $\bar{a}$ realized in the structure $\mc{M}$,
		\[ P_{\bar{a}} = \{ p_{\bar{a}\bar{b}}(\bar{x},\bar{y}) : \bar{b}\in \mc{M}\}.\]
		Let $\overline{P_{\bar{a}}}$ be the set of all $\Pi_{\lambda}$ types realised in $\mc{M}$ but not in $P_{\bar{a}}$.
		Consider the sentence
		\[ \theta = \bigdoublewedge_{\bar{a} \in \mc{M}} \exists \bar{x} \bigdoublewedge_{p \in \overline{P_{\bar{a}}}} \neg \exists \bar{y} \;  p(\bar{x},\bar{y}).\]
		Note that $\theta$ is true in $\mc{M}$ by construction and that it is a $\Pi_{\lambda+3}$ formula. In particular, $\theta$ is a conjunction of $\Sigma_{\lambda+2}$ formulas.
		If $\mc{M} \geq_{\lambda+2} \mc{N}$, then since $\theta$ is true in $\mc{M}$, each of its $\Sigma_{\lambda+2}$ conjuncts is true in $\mc{N}$.
		This means that $\theta$ is true in $\mc{N}.$
		(Alternatively, one can note that $\theta$ is $\vwE_\alpha$ and appeal to Proposition \ref{prop:transfer-over-bf} which we prove later.) Thus if $\mc{M} \geq_{\lambda+2} \mc{N}$ then $\mc{N} \models \psi \wedge \theta$.
		
		On the other hand, suppose that $\mc{N} \models \psi \wedge \theta$. Then $\mc{M}$ and $\mc{N}$ realise the same $\Pi_{\lambda}$ types. Given $\bar{a}\in \mc{M}$, choose $\bar{b} \in \mc{N}$ such that
		\[ \mc{N} \models \bigdoublewedge_{p \in \overline{P_{\bar{a}}}} \neg \exists \bar{y} \;  p(\bar{b},\bar{y}).\]
		Now given $\bar{b}' \in \mc{N}$, we have that
		\[ \mc{N} \models \bigdoublewedge_{p \in \overline{P_{\bar{a}}}} \neg p(\bar{b},\bar{b}').\]
		But the $\Pi_\lambda$ types realised in $\mc{N}$ are the same as the $\Pi_{\lambda}$-types realised in $\mc{M}$ because $\mc{N}\models\psi$, and so since $\bar{b}\bar{b}'$ does not realise one of the types in $\overline{P_{\bar{a}}}$ it must realise one of the types in $P_{\bar{a}}$, say $p_{\bar{a}\bar{a}'}$. Then $(\mc{M},\bar{a}\bar{a}') \equiv_\lambda (\mc{N},\bar{b}\bar{b}')$. Working backwards through the proof gives that $(\mc{M},\bar{a})\geq_{\lambda+1}(\mc{N},\bar{b})$ and so $\mc{M} \geq_{\lambda+2} \mc{N}$. Thus if $\mc{N} \models \psi \wedge \theta$ then $\mc{M} \geq_{\lambda+2} \mc{N}$.
	\end{proof}
	
	\section{Lower bounds for the definability of the $\alpha$-types of structures}
	
	We now consider lower bounds for the complexity of the sets $\{ \mc{N} : \mc{M} \leq_\alpha \mc{N}\}$ and $\{ \mc{N} : \mc{M} \geq_\alpha \mc{N}\}$.
	We only consider particular values of $\alpha$ and will show how to use these results to prove the complete characterization in Section \ref{sec:jump-inv}.
	We begin by considering small finite values for which the back-and-forth relations are very easy to define.
	We then look at slightly larger finite values that exhibit more general behavior and are therefore suitable base cases for the jump inversion argument in Section \ref{sec:jump-inv}.
	Lastly, we analyze exceptional cases near limit ordinals. In all cases, we have already shown that the sets are of the specified complexity, and so it is just the hardness results that remain.
	
	\subsection{Lower bounds for small finite cases}
	The following three cases are too small to exhibit general behavior, so they are treated separately.

	\
	\begin{lemma}
		There is a structure $\mc{M}$ such that the set
		\[ \{ \mc{N} : \mc{N} \geq_1 \mc{M} \} \]
		is $\bfPi^0_{1}$-complete.
	\end{lemma}
	
	\begin{proof}
		Let $\mc{M}$ be a structure in the language of one unary relation $U$ where $U$ holds of exactly one element. Given a $\bfPi^0_{1}$ set $A$, we can build for each $x \in 2^\omega$ a structure $\mc{N}^x$ such that $U$ holds of exactly one element in $\mc{N}^x$ (and hence $\mc{N}^x \geq_1 \mc{M}$) if $x \in A$ and $U$ holds of exactly two elements in $\mc{N}^x$ (and hence $\mc{N}^x \ngeq_1 \mc{M}$) if $x \notin A$.
	\end{proof}

	\begin{lemma}\label{lem:Hard1}
		There is a structure $\mc{M}$ such that the set
		\[ \{\mc{N} : \mc{N} \leq_1 \mc{M}\}\]
		is $\bfPi^0_2$-complete.
	\end{lemma}
	
	\begin{proof}
		Let $\mc{M}$ be the structure in the language of one unary relation $U$ where $U$ holds of an infinite-coinfinite subset of the elements of $\mc{M}$. The set $\Inf \subseteq 2^\omega$ of strings with infinitely many 1's is $\bfPi^0_2$-complete (see e.g., Section 23.A of \cite{Kechris} for this and other examples of complete sets), so we will reduce it to $\{\mc{N} : \mc{N} \leq_1 \mc{M}\}$. Given $x \in 2^\omega$ we construct $\mc{N}^x$ so that $x\in \Inf \iff \mc{N}^x\leq_1 \mc{M}$. If $\mc{N}^x$ has domain $\mathbb{N}$, let $\mc{N}^x\models U(2i)$ if and only if $x(i)=1$ ($U$ will never hold of any odd elements).
		If $x$ has infinitely many 1's then $\mc{N}^x\cong\mc{M}$, so in particular $\mc{N}^x\leq_1 \mc{M}$.
		If $x$ has finitely many 1's, say $n$ of them, then in $\mc{M}$ at least $n+1$ distinct elements satisfy the relation $U$, but in $\mc{N}^x$ there are only $n$. Thus $\mc{N}^x\not\leq_1 \mc{M}$ as desired.
	\end{proof}
	
	\begin{lemma}\label{lemma:2hardeasy}
		There is a structure $\mc{M}$ such that the set
		\[ \{\mc{N} : \mc{N} \leq_2 \mc{M}\}\]
		is $\bfPi^0_3$-complete.
	\end{lemma}
	\begin{proof}
		Identify $(2^\omega)^\omega$ with all lists $(a_i)_{i \in \omega}$ of infinite binary strings. Let $A \subseteq (2^\omega)^\omega$ be the set of lists $(a_i)_{i \in \omega}$ such that every $a_i$ has finitely many $1$'s. This set is $\bfPi^0_3$-complete.
		
		Consider structures of the following type: they have countably many sorts, labeled by unary relations $R_i$, each sort consisting of a linear ordering.
		Given a structure $\mc{C}$ of this form, we write $\mc{C}_i$ to mean the linear ordering defined by the $R_i$ points.
		We let $\mc{M}$ be the unique structure of this form with $\mc{M}_i\cong\omega$ for all $i$.
		Given $x = (x_i)_{i \in \omega}$ we construct in a continuous way $\mc{N}^x$, also a structure of the above form, so that $x\in A \iff \mc{N}^x\leq_2 \mc{M}$.
		Given $i$, we explain how to construct $\mc{N}^x_{i}$ the $i$th sort of $\mc{N}^x$.
		Start with a copy of $\omega$.
		Whenever we see another 1 in the string $x_i$, add an element below all of the previously added elements.
		If $x_i$ has finitely many 1's, then in $\mc{N}^x_i$ we add only finitely many elements below the initial copy of $\omega$, and so $\mc{N}^x_i \cong \omega$.
		If $a_i$ has infinitely many 1's, then $\mc{N}^x_i=\omega^*+\omega\cong\zeta$.
		If $x\in A$, all of the $x_i$ have finitely many 1's and $\mc{N}^x\cong\mc{M}$, so in particular $\mc{N}^x \leq_2 \mc{M}$.
		On the other hand, if $x\not\in A$, some $x_i$ has infinitely many 1's, so $\mc{N}^x_i\cong\zeta$.
		Let
		\[ \varphi:=\exists y\; \left[ R_i(y) \land \forall z \left( R_i(z)\longrightarrow  z\geq y \right) \right].\]
		Note that $\mc{M}\models \varphi$, yet $\mc{N}^x\models \neg\varphi$, so $\mc{N}^x\nleq_2 \mc{M}$ as desired.
	\end{proof}
	
	\subsection{Critical finite cases}
	
	The next two cases we consider are critical because they exhibit all of the complexity of the behavior of the general case. In Section \ref{sec:jump-inv}, we will use jump inversion to obtain all of the other non-exceptional cases from these two.
	
	\begin{lemma}\label{lem:2Hard}
		There is a structure $\mc{M}$ such that the set
		\[ \{ \mc{N} : \mc{N} \geq_2 \mc{M} \} \]
		is $\bfPi^0_{4}$-complete.
	\end{lemma}
	\begin{proof}
		The structures we work with will be \textit{flower graphs} (sometimes also called bouquet graphs). Let $\mc{S}$ be a non-empty family of subsets of $\omega$. (From now on, all of our families will be assumed to be non-empty.) There is a corresponding flower graph $\mc{G}_{\mc{S}}$. This flower graph will have, for each $S \in \mc{S}$, infinitely many connected components (all isomorphic to each other) consisting of a central vertex with, for each $n \in S$, a loop of length $n$.
		
		This graph encodes what we will call \textit{positive enumerations} of the family $\mc{S}$. A positive enumeration of $\mc{S}$ should be an object from which positive information about $\mc{S}$ can be obtained positively, though one has to express this formally. There are many equivalent definitions, but one is as follows. A positive enumeration of a family is a function $f \colon \omega \times \omega \times \omega \to \{0,1\}$ such that, for each $x,y$, $f(x,y,0) = 0$ and if $f(x,y,s) = 1$ then $f(x,y,t) = 1$ for all $t \geq s$. We set
		\[ \hat{f}(x,y) = \lim_{s \to \infty} f(x,y,s).\]
		Let
		\[ X_i = \{ j : \hat{f}(i,j) = 1\}.\]
		Then $f$ is a positive enumeration of the family $\{ X_i : i \in \omega\}$. Note that the order in which the subsets of $\omega$ appear does not matter, and sets can be repeated. The same set can show up multiple times in the enumeration.
		
		The idea is that $\mc{G}_{\mc{S}}$ has the property, easily verified, that (a) from any enumeration of $\mc{S}$, we can compute a copy of $\mc{G}_{\mc{S}}$, and (b) from any copy of $\mc{G}_{\mc{S}}$ we can enumerate a copy of $\mc{S}$. (In particular, there is a continuous map that takes an enumeration of $\mc{S}$ to a presentation of $\mc{G}_{\mc{S}}$, and a continuous map taking a presentation of $\mc{G}_{\mc{S}}$ to an enumeration of $\mc{S}$.)
		
		We say that a family $\mc{S}$ is \textit{closed under finite additions of elements} if whenever $S \in \mc{S}$ and $F$ is finite, then $S \cup F \in \mc{S}$.
		
		\begin{claim}\label{claim:subsetleq2}
			Let $\mc{S}$ and $\mc{T}$ be two families of subsets of $\omega$, and suppose that both $\mc{S}$ and $\mc{T}$ are closed under finite additions of elements. Then $\mc{G}_{\mc{S}} \leq_2 \mc{G}_{\mc{T}}$ if and only if for each $T \in \mc{T}$ there is $S \in \mc{S}$ such that $S \subseteq T$. 
		\end{claim}
		\begin{proof}
			Suppose that there is $T \in \mc{T}$ such that for all $S \in \mc{S}$ we have $S \nsubseteq T$. We describe a winning strategy for Spoiler in the back-and-forth game to witness $\mc{G}_{\mc{S}} \nleq_2 \mc{G}_{\mc{T}}$. First, Spoiler plays the central vertex $u \in \mc{G}_{\mc{T}}$ of a connected component corresponding to $T$. Say that Duplicator responds with the central vertex $v \in \mc{G}_{\mc{S}}$ of a connected component corresponding to $S \in \mc{S}$. Since $S \nsubseteq T$, there is some $n \in S$ with $n \notin T$. Then Spoiler plays an $n$-cycle in $\mc{G}_{\mc{S}}$ in the connected component of $v$, to which Duplicator cannot respond with an $n$-cycle in the connected component of $u$. If Duplicator responds with a vertex $v \in \mc{G}_{\mc{S}}$ that is not a central vertex, then Spoiler can play a path connecting $v$ to its central vertex $v'$ along with an $n$-cycle connected to $v'$. Again, Duplicator cannot respond with an $n$-cycle in the connected component of $u$. Thus $\mc{G}_{\mc{S}} \nleq_2 \mc{G}_{\mc{T}}$.
			
			Now suppose that for each $T \in \mc{T}$ there is $S \in \mc{S}$ such that $S \subseteq T$. We will describe a winning strategy for Duplicator in the back-and-forth game to witness  $\mc{G}_{\mc{S}} \leq_2 \mc{G}_{\mc{T}}$. On their first play, Spoiler plays elements of $\mc{G}_{\mc{T}}$; we may assume that they play central vertices $u_1,\ldots,u_k$ together with, for each $i$, the elements of finitely many loops in the connected component of $u_i$. For each $i$, let $F_i$ be the finite set of sizes of loops that Spoiler plays connected to $u_i$. For each $i$, let $T_i \in \mc{T}$ be the set coded by the connected component of $u_i$, and choose in $\mc{S}$ a set $S_i$ with $F_i \subseteq S_i \subseteq T_i$. Duplicator will respond, in $\mc{G}_{\mc{S}}$, with vertices $v_1,\ldots,v_k$ chosen such that the connected component of $v_i$ codes $S_i \in \mc{S}$; using the fact that $F_i \subseteq S_i$, Duplicator also responds to each loop played by Spoiler on $u_i$ with a loop of the same size on $v_i$. Now in the second round of the game suppose that Spoiler plays, in $\mc{G}_{\mc{S}}$, finitely many further loops on the $v_i$, together with new central vertices $v_1',\ldots,v_\ell'$ and finitely many loops in their connected components. Since $S_i \subseteq T_i$, Duplicator can respond to each loop on $v_i$ with a loop on $u_i$ of the corresponding size. For each $j$ let $G_j$ be the set of sizes of loops played by Spoiler on $v_j'$, and choose $T_j \supseteq G_j$. Then Duplicator can respond to each $v_j'$ with a central vertex $u_j'$ coding $T_j$, and to the loops played on $v_j'$ with loops of the same lengths on $u_j'$. Thus $\mc{G}_{\mc{S}} \leq_2 \mc{G}_{\mc{T}}$.
		\end{proof}
		
		Let $\mc{S} = \{S_i \cup F: i \in \mathbb{N}, F \in [\mathbb{N}]^{< \omega}\}$ where $S_i = \{ \la i,j \ra \}_{j\in\omega}$. 
		
		\begin{claim}\label{claim:pi4hard}
			The set of enumerations $f$ of families $\mc{T}$ such that for all $T \in \mc{T}$ there is $S \in \mc{S}$ with $S \subseteq T$ is $\bfPi^0_4$-complete. Moreover, given any $\bfPi^0_4$ set $A$, there is a Wadge reduction $x \mapsto f_x$ such that each $f_x$ is an enumeration of a family which is closed under finite additions of elements.
		\end{claim}
		\begin{proof}
			Let $A$ be the $\bfPi^0_4$-complete set of all lists $f = (f_{u,v})_{u,v \in \omega}$ of graphs of functions $f_{u,v}: \omega \to \omega$, with
			\[f \in A \Longleftrightarrow \forall u \exists v \; f_{u,v} \text{ is total}.\]
			Given $f = (f_{u,v})$, we produce an enumeration $g$ of the family
			$\mc{T} = \{ T_{u,F} : u \in \mathbb{N}, F \in [\mathbb{N}]^{< \omega}\}$ where
			\[ T_{u,F} = F \cup \{\la v,j \ra : v \in \mathbb{N} \text{ and }  f_{u,v}(j)\downarrow\}.\]
			We must show that $f \in A$ if and only if for all $T \in \mc{T}$ there is $S \in \mc{S}$ such that $S \subseteq T$.
			
			Say that $f\in A$. In this case, we must show that there is a pair $i \in \mathbb{N}, F' \in [\mathbb{N}]^{< \omega}$ such that $T_{u,F}\supseteq S_i\cup F'$. Let $F=F'$ and take $i$ to be the witness showing that $f_{u,i}$ is total. In particular this means that $ f_{u,i}(j)\downarrow$ always holds so for all $j$, $\la i,j \ra\in T_{u,F}$.
			It follows at once that $T_{u,F}\supseteq S_i\cup F'$ as desired.
			
			Conversely say that there is an $S_i \in \mc{S}$ such that $S_i \subseteq T_{u,\emptyset}$.
			This means that for all $j$, $ f_{u,i}(j)\downarrow$, that is, $f_{u,i}$ is total.
			This holds for any $u$, so we have that $f\in A$, as desired.
		\end{proof}
		
		Together, these two claims finish the proof. We take $\mc{M} = \mc{G}_{\mc{S}}$. We reduce to $\{ \mc{N} : \mc{N} \geq_2 \mc{M}\}$ any $\bfPi^0_4$ set $B$ by producing, for each $x$, an enumeration $f_x$ of a family $\mc{T}_x$ which is closed under finite additions of elements, and from this a presentation of $\mc{G}_{\mc{T}_x}$. Then $\mc{G}_{\mc{T}_x} \geq_2 \mc{G}_{\mc{S}}$ if and only if $x \in B$.
	\end{proof}
	
	\begin{lemma}\label{lem:3hard}
		There is a structure $\mc{M}$ such that
		\[ \{\mc{N} : \mc{N} \leq_3 \mc{M}\}\]
		is $\bfPi^0_6$-complete.
	\end{lemma}
	\begin{proof}
		We will make use of the construction of Lemma \ref{lem:2Hard}. Let $\mc{S} = \{S_i \cup F: i \in \mathbb{N}, F \in [\mathbb{N}]^{< \omega}\}$, where $S_i = \{ \la i,j \ra \}_{j\in\omega}$, be the family of subsets of $\omega$ defined in Lemma \ref{lem:2Hard} and let $\mc{S}_{\fin}$ be the family of all finite subsets of $\omega$.
		
		In addition to what we proved earlier, we will need two additional facts. First, in Claim \ref{claim:pi4hard} we can replace the family $\mc{T}$ by $\mc{T} \cup \mc{S}$ to assume that $\mc{T}$ contains $\mc{S}$. This modification does not change whether for all $T \in \mc{T}$ there is $S \in \mc{S}$ with $S \subseteq T$. Second, under this assumption, we can improve one direction of Claim \ref{claim:subsetleq2} as follows.
		
		\begin{claim}\label{claim:geq3copy}
			Let $\mc{S}$ and $\mc{T}$ be two families of subsets of $\omega$, and suppose that both $\mc{S}$ and $\mc{T}$ are closed under finite additions of elements. Suppose that $\mc{S} \subseteq \mc{T}$ and that for each $T \in \mc{T}$ there is $S \in \mc{S}$ such that $S \subseteq T$. Then $\mc{G}_{\mc{S}} \geq_3 \mc{G}_{\mc{T}}$.
		\end{claim}
		\begin{proof}
			We give a strategy for Duplicator in the back-and-forth game to witness $\mc{G}_{\mc{S}} \geq_3 \mc{G}_{\mc{T}}$.
			On their first play, Spoiler plays elements $\bar{p}\in\mc{G}_{\mc{S}}$ intersecting connected components coding $S_1,\ldots,S_k$.
			Because $S \subseteq T$, there is a natural embedding $\iota: \mc{G}_{\mc{S}}\to \mc{G}_{\mc{T}}$.
			Duplicator plays $\iota(\bar{p})$ in response.
			We claim that $(\mc{G}_{\mc{S}},\bar{p}) \leq_2 (\mc{G}_{\mc{T}},\iota(\bar{p}))$, so this is a winning move.
			Duplicator will play according to $\iota$ on the connected components coding $S_1,\ldots,S_k$ in each structure.
			This is a winning strategy as $\iota$ restricts to an isomorphism on these components.
			The remaining connected components on each side only finitely differ from $\mc{G}_{\mc{S}}$ and $\mc{G}_{\mc{T}}$.
			In particular, they are isomorphic to the whole of $\mc{G}_{\mc{S}}$ and $\mc{G}_{\mc{T}}$.
			This means that Duplicator can play a winning strategy for $\mc{G}_{\mc{S}} \leq_2\mc{G}_{\mc{T}}$ on these components to win the $(\mc{G}_{\mc{S}},\bar{p}) \leq_2 (\mc{G}_{\mc{T}},\iota(\bar{p}))$ game.
			Under the assumptions provided, such a winning strategy is provided in Claim \ref{claim:subsetleq2}.
		\end{proof}
		
		The following general claim will allow us to conclude the desired result.
		
		\begin{claim}\label{claim:reverseplus2}
			Let $\mc{M}$ be a structure such that for all  $\bfPi^0_\beta$ sets $C$ there is a continuous map $x\mapsto \mc{N}_x$ such that:
			\begin{enumerate}
				\item if $x\in C$ then $\mc{N}_x\leq_{\alpha+1} \mc{M}$,
				\item if $x\not\in C$ then $\mc{N}_x\not\geq_{\alpha} \mc{M}$, and
				\item there is a structure $\mc{F}$ such that for all $x$ $\mc{F}\leq_\alpha N_x$ and $\mc{F}\not\geq_\alpha \mc{M}$.
			\end{enumerate}
			Then there is a structure $\widetilde{\mc{M}}$ such that $\{\widetilde{\mc{N}} : \widetilde{\mc{N}} \leq_{\alpha+1} \widetilde{\mc{M}}\}$ is $\bfPi^0_{\beta+2}$-hard.
		\end{claim}
		
		\begin{proof}
			Let $A$ be a $\bfPi^0_{\beta+2}$ set, and write
			\[ x \in A \Longleftrightarrow \forall m \exists n \; (x,m,n) \in B\]
			where $B$ is $\bfPi^0_\beta$.
			By hypothesis for each $m,n$ there is a structure $\mc{N}^x_{m,n}$ such that
			\[ (x,m,n) \in B \Longrightarrow \mc{N}^x_{m,n} \leq_{\alpha+1} \mc{M}\]
			and
			\[ (x,m,n) \notin B \Longrightarrow  \mc{N}^x_{m,n} \ngeq_\alpha \mc{M}.\]
			Fix also the structure $\mc{F}$. We must define a structure $\widetilde{M}$ and build, for each $x$,
			a structure $\widetilde{\mc{N}}^x$ such that
			\[ x \in A \Longleftrightarrow \widetilde{\mc{N}}^x \leq_{\alpha+1} \widetilde{\mc{M}}.\]
			We work in the language of $\mc{M}$ together with infinitely many fresh unary relations $U_i$ and a fresh equivalence relation $E$.
			Our structures $\widetilde{\mc{M}}$ and $\widetilde{\mc{N}}^x$ will be divided by the $U_i$ into infinitely many distinguished sorts, each of which consists of infinitely many infinite $E$-equivalence classes, with the equivalence classes refining the sorts.
			On each equivalence class, there will be a structure in the language of $\mc{M}$; the choice of structures that appear in each equivalence class of each sort will be the only difference between the different structures.
			For $\widetilde{\mc{M}}$, in each sort $U_n$ we have one equivalence class which is a copy of $\mc{M}$ and each other equivalence class is a copy of $\mc{F}$.
			For $\widetilde{\mc{N}}^x$, in each sort $U_m$, we put as the $2n$th equivalence class a copy of $\mc{N}^x_{m,n}$, and as the $2n+1$th equivalence class a copy of $\mc{F}$.
			
			We claim that $\widetilde{\mc{N}}^x \leq_{\alpha+1} \widetilde{\mc{M}}$ if and only if $x \in A$.
			First, suppose that $\widetilde{\mc{N}}^x \leq_{\alpha+1} \widetilde{\mc{M}}$.
			Given $m$, let $u \in \widetilde{\mc{M}}$ be an element of the sort $U_m$ and within $U_m$ of the equivalence class on which we have a copy of $\mc{M}$.
			There must be $v \in \widetilde{\mc{N}}^x$ such that $(\widetilde{\mc{N}}^x,v) \geq_{\alpha} (\widetilde{\mc{M}},u)$; moreover, $v$ must be in the $m$th sort.
			If $G_u$ is the structure on the equivalence class of $u$, and $G_v$ is the structure on the equivalence class of $v$, then $G_v \geq_{\alpha} G_u$.
			Then $v$ cannot be in the $2n+1$th equivalence class of the $m$th sort of $\mc{N}_x$, since $\mc{F}\not\geq_\alpha \mc{M}$.
			Thus $v$ must be in the $2n$th equivalence class for some $n$, and $G_v \cong \mc{N}^x_{m,n} \geq_\alpha \mc{M}$. 
			By assumption, this implies that $(x,m,n) \in B$.
			So for each $m$ there is $n$ such that $(x,m,n) \in B$, which implies that $x \in A$.
			
			On the other hand, suppose that $x \in A$.
			For each $m$ there is $n$ such that $(x,m,n) \in B$, and so for each $m$ there is $n$ such that $\mc{N}^x_{m,n} \leq_{\alpha+1} \mc{M}$.
			We must show that $\widetilde{\mc{N}}^x \leq_{\alpha+1} \widetilde{\mc{M}}$.
			Since $\mc{M}$ is divided into sorts by the $U_i$, it suffices to show that this is true on each sort.
			Consider the $m$th sort, and fix $n$ such that $\mc{N}^x_{m,n} \leq_{\alpha+1} \mc{M}$.
			For each equivalence class of $\widetilde{\mc{M}}$ (with induced structure $G$), there is an equivalence class of $\widetilde{\mc{N}}^x$ (with induced structure $H$) such that $G \geq_{\alpha+1} H$; for the first equivalence class where $G \cong \mc{M}$ this is the $2n$th equivalence class with $H \cong \mc{N}^x_{m,n}$, and for any other equivalence class with $G \cong \mc{F}$ this is any of the infinitely many equivalence classes with $H \cong \mc{F}$.
			Also, for each equivalence class of $\widetilde{\mc{N}}^x$ (with induced graph $H \cong \mc{N}^x_{m,n}$ or $H \cong \mc{F}$) there are infinitely many equivalence classes of $\mc{M}$ (with induced structure $G \cong \mc{F}$) such that $H \geq_{\alpha} G$.
			This is enough to check that $\widetilde{\mc{N}}^x \leq_{\alpha+1} \widetilde{\mc{M}}$ (see Sublemma \ref{sublemma:complicated}).
		\end{proof}
		
		We now complete the proof of the lemma by using Claim \ref{claim:reverseplus2} with $\alpha=2$.
		By Claim \ref{claim:geq3copy} along with the construction in Lemma \ref{lem:2Hard}, there is a continuous map $x\mapsto \mc{N}_x$ satisfying the first two conditions of Claim \ref{claim:reverseplus2}.
		Therefore, to finish the proof, we need only to construct an $\mc{F}$ that satisfies the last condition of Claim \ref{claim:reverseplus2} relative to this map.
		Let $\mc{F}=\mc{G}_{\mc{S}_{\fin}}$.
		Note that $\mc{N}_x\cong \mc{G}_{\mc{T}}$ where $\mc{T}$ is always closed under finite additions of elements.
		This means that for all finite sets $F$, there is a $T\in \mc{T}$ such that $F\supseteq T$.
		By Claim \ref{claim:subsetleq2}, this means that $\mc{F}\leq_2 \mc{N}_x$.
		Also, $\mc{M}=\mc{G}_{\mc{S}}$ and $\mc{S}$ has only infinite sets; none of these sets are contained in any finite set, so by Claim \ref{claim:subsetleq2}  $\mc{M}\not\leq_2 \mc{F}$, as desired.
		This completes the proof of the lemma.
	\end{proof}
	
	\subsection{Exceptional lower bounds at limit levels}
	The following cases are too close to a limit ordinal to exhibit general behavior, so they are considered separately.
	
	\begin{lemma}
		If $\alpha$ is a limit ordinal, $\mc{M}\leq_\alpha \mc{N}$ if and only if $\mc{M}\geq_\alpha \mc{N}$.
		There is a structure $\mc{M}$ such that the set
		\[ \{\mc{N} : \mc{N} \equiv_\alpha \mc{M}\}\]
		is $\bfPi^0_{\alpha}$-complete.
	\end{lemma}
	
	\begin{proof}
		The first claim is well-known; see for example the proof of Lemma II.67 in \cite{MonBook}.
		
		To see the second claim, let $\delta_i\to\alpha$ be a fundamental sequence.
		By the Ash and Knight Pair of Structures theorem \cite{AK90} the set of models isomorphic to $\omega^{\delta_i}$ is $\bfPi^0_{\delta_i}$-hard.
		Let the language of $\mc{M}$ be $\{<\}$ along with infinitely many unary relations $R_i$.
		Let $\mc{M}$ contain exactly one copy of each $\omega^{\delta_i}$, where $R_i$ holds exactly of the elements in  $\omega^{\delta_i}$.
		
		Given a $\bfPi^0_\alpha$ set $A$, we can write $A=\bigcap_i A_i$ where $A_i$ is $\bfPi^0_{\delta_i}$.
		For $x\in2^\omega$, let $\mc{N}^x_{i}\cong \omega^{\delta_i}$ if and only if $x\in A_i$.
		Note that this can be done in a continuous fashion.
		$\mc{N}^x$ contains exactly a copy of each $\mc{N}^x_{i}$ and $R_i$ holds exactly of the elements in $\mc{N}^x_{i}$.
		As each component of $\mc{N}^x$ is constructed continuously, so is all of $\mc{N}^x$.
		If $x\in A$ it is immediate that $\mc{M}\cong \mc{N}^x$ and so $\mc{M}\equiv_\alpha \mc{N}$.
		If $x\not\in A$, then for some $i$ $\mc{N}^x_{i}\not\cong \omega^{\delta_i}$.
		In particular, $\mc{N}^x$ fails to satisfy the $\Pi_{2\delta_i+1}$ Scott sentence of $\omega^{\delta_i}$ (give explicitly in \cite{MonBook} Lemma II.5) relativized to the predicate $R_i$.
		Thus, $\mc{N}^x\not\equiv_\alpha \mc{M}$, completing the proof.
	\end{proof}

	For the next case, we need to make use of Montalb\'an's tree of structures theorem as seen in \cite{MonAsh},
	We only need a specific case of the theorem (namely, the one where the ``tree'' has only one path that is the length of a limit ordinal), so we will only state the weak version needed.
	
	\begin{theorem}[Montalb\'an \cite{MonAsh}]\label{thm:MonTree}
		Given a limit ordinal $\lambda$ and a fundamental sequence $\delta_i\to\lambda$, let $\{\mc{L}_i\}_{i\in\omega}$ and $\mc{L}_\omega$ be structures with $\mc{L}_i\equiv_{\delta_i+1} \mc{L}_{i+1} \equiv_{\delta_i+1} \mc{L}_\omega$ for all $i$.
		Given a $\bfPi^0_\lambda$ set $A=\bigcap_i A_i$ where each $A_i$ is $\bfPi^0_{\delta_i}$ there is a continuous procedure that given $x$ outputs $\mc{L}_i$ if $i$ is the least such that $x\not\in A_i$ and $\mc{L}_\omega$ if $i\in A$.
	\end{theorem}
	
	Note that examples of such $\mc{L}_i$ with the additional property that $M_i\not\equiv_{\delta_i+2}\mc{L}_{i+1}$ always exist; see for example, \cite{MonBook} Example IX.24 where the examples are linear orderings.
	
	\begin{lemma}\label{lem:succlimit}
		If $\alpha=\lambda+1$ where $\lambda$ is a limit ordinal, there is a structure $\mc{M}$ such that the set
		\[ \{\mc{N} \; : \; \mc{N} \geq_\alpha \mc{M}\}\]
		is $\bfPi^0_{\alpha+1}$-complete.
	\end{lemma}
	
	\begin{proof}
		We first construct an example $\mc{K}$ such that  $\{\mc{N} \; : \; \mc{N} \geq_\alpha \mc{K}\}$ is $\bfSigma^0_{\lambda+1}$ complete; we will use such a $\mc{K}$ as a piece of our final construction.
		Let $\{\mc{L}_i\}_{i\in\omega}$ and $\mc{L}_\omega$ be linear orderings with $\mc{L}_i\equiv_{\delta_i+1} \mc{L}_{i+1} \equiv_{\delta_i+1} \mc{L}_\omega$ and $\mc{L}_i\not\equiv_{\delta_i+2}\mc{L}_{i+1}$ for all $i$.
		$\mc{K}$ will be a structure in the signature of two equivalence relations $E$ and $F$ and an ordering $<$.
		The equivalence relation $F$ will refine $E$.
		Given some $\sigma\in(\omega+1)^{\omega}$, let $\mc{K}_\sigma$ be the structure in the language $\{F,<\}$ with infinitely many $F$ equivalence classes, each of which is a linear order. These linear orders will be infinitely many copies of each $\mc{L}_{\sigma(j)}$ for each $j \in \omega$.
		Then define $\mc{K}$ as follows with infinitely many $E$ equivalence classes, each of which is a structure $\mc{K}_\sigma$. For each $\sigma\in(\omega+1)^{\omega}$ with only finitely many finite values, $\mc{K}$ will have infinitely many equivalence classes with a copy of $\mc{K}_\sigma$.
		
		Let $B$ be a $\bfSigma^0_{\lambda+1}$ set.
		Let $x\in B$ if and only if $\exists m ~ x\in C_m$ for a family of $\bfPi^0_{\lambda}$ sets $C_m$.
		Let $A_j=C_0\cup\cdots\cup C_j$.
		For each $j$, write $A_j=\bigcap_i A_{j,i}$ where each $A_{j,i}$ is $\bfPi^0_{\delta_i}$.
		Let $F_j$ be the continuous map from Theorem \ref{thm:MonTree} with respect to the set $A_j$ and let $\mc N^x_j := F_j(x)$.
		Note that $\bigsqcup_{j\in\omega} \mc N^x_{j}$ is isomorphic to $\mc{K}_{\tau_0}$ for some $\tau_0\in(\omega+1)^{\omega}$.
		Let $\{\tau_k\}_{k\in\omega}$ enumerate the set of elements in $(\omega+1)^{\omega}$ that differ from $\tau_0$ on only finitely many inputs.
		Let $\mc{N}^x$ contain infinitely many $E$ equivalence classes, countably many for each $\tau_k$ where the equivalence class corresponding to each $\tau_k$ contains a copy of $\mc{K}_{\tau_k}$.
		Note that $\mc{N}^x$ is constructed by countably many disjoint continuous modifications of $F_j$, and so in particular the map $x\mapsto \mc{N}^x$ is continuous.
		(To modify $F_k$, first designate countably many equivalence classes for each $p\in (\omega+1)^{<\omega}$; then on each equivalence class corresponding to $p$ put a copy of $\mc{L}_{p(\ell)}$ for $\ell\leq|p|$ and have $R_\ell$ hold of those points; for each $r>|p|$ put a copy of $N^x_r$ and have $R_r$ hold of those points.)
		
		We now show that $x\in B$ if and only if $\mc{N}^x\geq_{\lambda+1} \mc{K}$.
		If $x\in B$, eventually there is a witness $m$ such that $x\in C_m$.
		For $j\geq m$, this means that $x\in A_j$.
		In particular, $\tau_0(j)$ for $j\geq m$ is always $\omega$.
		This means that the set of elements of $(\omega+1)^{\omega}$ that are finitely different from $\tau_0$ are exactly those that have finitely many finite values.
		In particular, $\mc{N}^x\cong \mc{K}$, so $\mc{N}^x\geq_{\lambda+1} \mc{K}$ as required.
		Now say that $x\not\in B$.
		This means that for any $j$, $x\not\in A_j$.
		Thus $\tau_0$ has only finite values.
		Consider the formula $\psi:=\exists z ~ \bigwwedge_{i\in\omega} \varphi^{E}_{\tau_0(i)}(z)$, where $\varphi^E_{\tau_0(i)}(z)$ is the $\Pi_{< \lambda}$ sentence that expresses that in the $E$-equivalence class of $z$ there is an $F$-equivalence class $\delta_{\tau_0(i)}+2$-equivalent to $\mc{L}_{\tau_0(i)}$. 
		By construction, $\mc{N}^x\models \psi$; any element of an equivalence class corresponding to $\tau_0$ is a witness.
		However, $\mc{M}\models\lnot\psi$; among the $\mc{L}_j$ the $\delta_{\tau_0(i)}+2$-type of $\mc{L}_{\tau_0(i)}$ isolates it up to isomorphism and no $E$ class of $\mc{M}$ has all of the $\mc{L}_{\tau_0(i)}$ included (in fact, each such $E$ class only has finitely many  $\mc{L}_j$).
		As $\psi$ is a $\Sigma_{\lambda+1}$ formula we see that $\mc{N}^x\not\geq_{\lambda+1} \mc{K}$, as desired.
		
		Finally, we construct the desired $\mc{M}$. Let $D$ be a $\bfPi^0_{\lambda+2}$ set.
		Let $x\in D$ if and only if $\forall n ~ x\in B_n$ for a family of $\bfSigma^0_{\lambda+1}$ sets $B_n$.
		For each $n$ let $\mc{K}_n \cong \mc{K}$ and $\mc{N}^x_n$ be the structures constructed above for $B_n$.
		$\mc{M}$ will be a structure in the signature of $<$, the linear ordering within each $\mc{L}_i$, two equivalence relations $E$ and $F$ along with infinitely many unary predicates  $\{S_n\}_{n\in\omega}$ denoting sorts in the structure.
		$\mc{M}$ contains a copy of each $\mc{K}_n\cong\mc{K}$, with the predicate $S_n$ holding exactly of the elements of $\mc{K}_n$.
		Construct $\widetilde{\mc{N}}_x$ with a copy of each of the structures $\mc{N}^x_n$ for each $n$; let $S_n$ hold exactly of the elements in $\mc{N}^x_n$.
		It is immediate that $\widetilde{\mc{N}}^x\geq_{\lambda+1}\mc{M}$ exactly when $\mc{N}^x_n\geq_{\lambda+1}\mc{K}_n$ for all $n$.
		By the above analysis, this happens exactly when  $\forall n ~ x\in B_n$, that is, $x\in D$, and this procedure is continuous as desired.
	\end{proof}

	\section{Jump inversion}\label{sec:jump-inv}
	
	We now move the results from Lemmas \ref{lem:uniform}, \ref{lem:Hard1}, \ref{lemma:2hardeasy}, \ref{lem:2Hard}, and \ref{lem:3hard} up the Borel hierarchy. To do this, we use the method of jump inversion following the methods and notation from Chapter X.3 of \cite{MonBook}, which itself follows Goncharov, Harizanov, Knight, McCoy, R. Miller, and Solomon \cite{GHKMMS}.
	The results listed in that chapter are unsuitable for our direct use; they lack uniformity in their statement and discuss constructions of computable reductions on indices rather than continuous Wadge reductions.
	Nevertheless, the proof techniques presented in Chapter X.3 of \cite{MonBook} do have the desired level of uniformity, even if the results are not stated in that way, and the arguments relativize in a manner that is suitable for the construction of a Wadge reduction.
	We provide a discussion of these techniques and the notation associated with them.
	
	Fix  $\alpha < \omega_1$. In Chapter X.3 of \cite{MonBook}, one works effectively, but as we are working non-effectively, we generally do not need to worry about issues like ordinal presentations. However, if we wanted, we could work relative to a presentation of $\alpha$. Let $\mc{L}_0$ and $\mc{L}_1$ be rigid and uniformly boldface $\bfDelta^0_{\alpha+1}$ categorical (equivalently, $\Sigma_{\alpha+1}$-atomic or having a $\Pi_{\alpha+2}$ Scott sentence; see \cite{MonBook} Theorems II.23 and VII.21) structures with one binary relation $S$ which are $\equiv_\alpha$-equivalent but with $\mc{L}_0 \nleq_{\alpha+1} \mc{L}_1$ and $\mc{L}_1 \nleq_{\alpha+1} \mc{L}_0$.
	Such structures $\mc{L}_0$ and $\mc{L}_1$ exist by Lemma X.14 of \cite{MonBook} (but such examples were also constructed in \cite{AK90}).
	It follows from the Ash and Knight pair of structures theorem \cite{AK90} that for any $\bfDelta^0_{\alpha+1}$ set $U$, there is a continuous function $x \mapsto \mc{C}_x$ such that
	\[ x \in U \Longrightarrow \mc{C}_x \cong \mc{L}_0\]
	and
	\[ x \notin U \Longrightarrow \mc{C}_x \cong \mc{L}_1.\]
	This exact statement is not observed in the original paper of Ash and Knight, but it is a direct corollary; the first explicit published reference known to the authors for a boldface version of the theorem is \cite{HM} Theorem 2.2. (We also note that the boldface result follows with a small amount of extra work from Louveau and Saint Raymond's separation theorem \cite{LSR}.)
	
	We can use these structures $\mc{L}_0$ and $\mc{L}_1$ to define the following map of structures.
	\begin{definition}
		Given a graph $\mc{A}=(A,E)$ construct $\Phi_\alpha(\mc{A})=(C;A,R)$ as follows in the language with a unary relation $A$ and a 4-ary relation $R$.
		There is one $A$-point in $\Phi_\alpha(\mc{A})$ for each element in $\mc{A}$, and we identify them.
		The non-$A$-points in $\Phi_\alpha(\mc{A})$ are split into disjoint infinite sets $B_{a,b}$ for each pair of $a,b\in A$.
		Given a fixed $a,b \in A$, $R(a,b,\cdot,\cdot)$ will be a binary relation on $B_{a,b}$. If $E(a,b)$ then $(B_{a,b},R(a,b,\cdot,\cdot))$ will be isomorphic to $\mc{L}_0$ and if $\neg E(a,b)$ then $(B_{a,b},R(a,b,\cdot,\cdot))$ will be isomorphic to $\mc{L}_1$.
	\end{definition}
	Intuitively, one replaces each edge with a copy of $\mc{L}_0$ and non-edge with a copy of $\mc{L}_1$.
	In Chapter X.3 of \cite{MonBook}, it is proven that $\Phi_\alpha$ is an infinitary $\Sigma_1$ bi-interpretation with the $\alpha$-canonical structural jump of the image.
	We collect the relevant consequences of this claim along with other observations about this construction in the theorem below.
	
	\begin{theorem}\label{thm:BiInt}
		For each $\alpha\in\omega_1$, the map $\Phi_\alpha$ has the following properties:
		\begin{enumerate}
			\item For every $\Pi_{\alpha+\beta}$ formula $\varphi$ there is a $\Pi_{\beta}$ formula $\varphi^*$ such that $\mc{A}\models \varphi^* \iff \Phi_\alpha(\mc{A})\models \varphi$.
			\item For every $\Pi_{\beta}$ formula $\psi$ there is a $\Pi_{\alpha+\beta}$ formula $\psi_*$ such that $\mc{A}\models \psi \iff \Phi_\alpha(\mc{A})\models \psi_*$.
			\item $\mc{A}\leq_\beta\mc{B}$ if and only if $\Phi_\alpha(\mc{A})\leq_{\alpha+\beta}\Phi_\alpha(\mc{B})$.
		\end{enumerate}
	\end{theorem}

	\begin{proof}
		Using Karp's theorem \cite{Karp}, (3) follows directly from (1) and (2). Essentially, (1) and (2) follow directly from the results in Section 2.2 of \cite{MR23} along with Chapter X.3 of \cite{MonBook}.
		To provide more detail, it follows from the existence of the $\Sigma_1$ bi-interpretation and the pull-back theorem (Theorem XI.7 of \cite{MonBook}) that
		\begin{enumerate}
			\item For every $\Pi_{\beta}$ formula $\varphi$ there is a $\Pi_{\beta}$ formula $\varphi^{**}$ such that
			\[\mc{A}\models \varphi^{**} \iff \Phi_\alpha(\mc{A})^{(\alpha)}\models \varphi.\]
			\item For every $\Pi_{\beta}$ formula $\psi$ there is a $\Pi_{\beta}$ formula $\psi_{**}$ such that
			\[\mc{A}\models \psi \iff \Phi_\alpha(\mc{A})^{(\alpha)}\models \psi_{**}.\]
		\end{enumerate}
		From here, it is enough to show that
		\begin{enumerate}
			\item For every $\Pi_{\beta}$ formula $\varphi^{**}$ there is a $\Pi_{\alpha+\beta}$ formula $\varphi^{*}$ such that
			\[ \Phi_\alpha(\mc{A})\models \varphi^{*} \iff \Phi_\alpha(\mc{A})^{(\alpha)}\models \varphi^{**}.\]
			\item For every $\Pi_{\alpha+\beta}$ formula $\psi_{*}$ there is a $\Pi_{\beta}$ formula $\psi_{**}$ such that
			\[ \Phi_\alpha(\mc{A})\models \psi_{**} \iff \Phi_\alpha(\mc{A})^{(\alpha)}\models \psi_{*}.\]
		\end{enumerate}
		The first item follows immediately by the definition of the  $\alpha$-canonical structural jump.
		The second item is shown by transfinite induction, with the only interesting case being the base case.
		We point out the uniformity in Montalb\'an's proof in Chapter X.3 of \cite{MonBook}.
		In particular, over every possible $\mc{A}$ he adds the same countably many $\Pi_\alpha$ types to the language to Morleyize $\Phi_\alpha(\mc{A})$ and obtain $\Phi_\alpha(\mc{A})^{(\alpha)}$.
		As a result, the translation of $\Pi_{\alpha+1}$ formulas about the $\alpha$-canonical structural jump into $\Pi_1$ formulas about the base structure provided in \cite{MR23} Proposition 11 is uniform across all input structures and directly provides our desired base case.
	\end{proof}

	The following lemmas are the key technical inputs needed to iterate our previous results through the ordinals.
	
	\begin{lemma}\label{lem:functoralpha}
		Let $X$ be a Polish space. Given a $\bfSigma^0_{\alpha+1}$-measurable map $F:X\to \Mod(E)$, there is a continuous map $\Phi_\alpha(F):X\to \Mod(A,R)$ such that $\Phi_\alpha(F)(x) \cong \Phi_\alpha(F(x))$.
	\end{lemma}
	
	\begin{proof}
		Consider a basic clopen set $U_{m,n,i}$ in $\Mod(E)$ for $n,m \in \omega$ and $i \in \{0,1\}$. If $i = 0$, then this is the set of all structures with $m \not E n$, and if $i = 1$, then it is the set of all structures with $m \; E \; n$. Since $F$ is $\bfSigma^0_{\alpha+1}$-measurable, $C_{m,n,i} = F^{-1}(U_{m,n,i})$ must be $\mathbf{\Delta}_{\alpha+1}$.
		Let $G_{m,n}$ be a continuous map, coming from the pair of structures theorem, reducing $(C_{m,n,0},C_{m,n,1})$ to $(\mc{L}_0,\mc{L}_1)$. That is, for each $x \in X$,
		\[ m \not E^{F(x)} n \Longleftrightarrow x \in C_{m,n,0} \Longleftrightarrow G_{m,n}(x) \cong \mc{L}_0\]
		and
		\[ m \; E^{F(x)} \; n \Longleftrightarrow x \in C_{m,n,1} \Longleftrightarrow G_{m,n}(x) \cong \mc{L}_1.\]
		We now define the map $\Phi_\alpha(F)$. Given $x$, $\Phi_\alpha(F)(x)$ will be the structure in the language $\{A,R\}$ which has an infinite set $A$ which we identify with $\mathbb{N}$, and for $m,n \in A$ we have disjoint infinite sets $B_{m,n}$ with $R(m,n,\cdot,\cdot)$ defined on $B_{m,n}$ so that $B_{m,n}$ is isomorphic to $G_{m,n}(x)$. This is a continuous construction since each $G_{m,n}$ is continuous, and $\Phi_\alpha(F)(x) \cong \Phi_\alpha(F(x))$ by choice of the $G_{m,n}$.
	\end{proof}
	
	\begin{lemma}\label{lem:contJumpInversion}
		If a class of graphs $\mathbb{D}$ is $\bfPi_n^0$ hard as an invariant set in $\Mod(E)$ then the class of structures
		\[\Phi_\alpha(\mathbb{D}) := \{ \mc{M} : \text{$\mc{M} \cong \Phi_\alpha(\mc{G})$ for some $\mc{G}\in \mathbb{D}$}\}\]
		is $\bfPi_{\alpha+n}^0$ hard as an invariant set in $\Mod(A,R)$.
		Moreover, $\Phi_\alpha(\mathbb{D})$ is $\bfPi^0_{\alpha+n}$ hard within the image of $\Phi_\alpha$, that is, the maps witnessing these reductions always output structures isomorphic to $\Phi_\alpha(\mc{G})$ for some $\mc{G}$.
	\end{lemma}
	\begin{proof}
		Let $S$ be a $\bfPi^0_{\alpha+n}$ subset of a Polish space $X$.
		We describe a continuous reduction from $S$ to $\Phi_\alpha(\mathbb{D})$, which will immediately yield the desired result.
		Let $S$ be $\Pi^0_{\alpha+n}$ relative to the parameter $Y$.
		Apply a change of topology to $X$ to make all $\Delta^0_{\alpha+1}(Y)$ sets clopen.
		In particular, $S$ is a $\bfPi_n^0$ set in this new topology.
		Let $G$ be the continuous map reducing $S$ to $\mathbb{D}$.
		Returning to the viewpoint of the original topology, $G: X \to \Mod(E)$ is a $\bfSigma^0_{\alpha+1}$-measurable map.
		By Lemma \ref{lem:functoralpha}, there is a continuous map $\Phi_{\alpha}(G): X \to \Mod(A,R)$ such that $\Phi_\alpha(G)(x) \cong \Phi_\alpha(G(x))$.
		Note that
		\[x\in S \iff G(x)\in\mathbb{D} \iff \Phi_\alpha(G(x))\in \Phi_\alpha(\mathbb{D}) \iff \Phi_\alpha(G)(x)\in \Phi_\alpha(\mathbb{D}).\]
		Therefore, $\Phi_\alpha(G)$ is the required continuous map witnessing the reduction.
	\end{proof}
	
	We are now ready to generalize Lemmas \ref{lem:2Hard} and \ref{lem:3hard} which give the general-case lower bounds.
	
	\begin{lemma}\label{cor:abovealpha}
		For $\alpha$ not a limit ordinal or successor of a limit ordinal, there is a structure $\mc{M}$ such that the set
		\[\{\mc{N} \; : \; \mc{N} \geq_\alpha \mc{M}\}\] is $\bfPi^0_{\alpha+2}$-complete.
	\end{lemma}	
	\begin{proof}
		Let $\mc{M}$ be the graph from Lemma \ref{lem:2Hard} such that $\{\mc{N} : \mc{N} \geq_2 \mc{M}\}$ is  $\bfPi^0_{4}$-complete. If $\alpha = 2$, we are done; otherwise, if $\alpha=\beta+2> 2$, we argue that for the jump inverse $\Phi_\beta(\mc{M})$ we have that  $\{\mc{N} : \mc{N} \geq_\alpha \Phi_\beta(\mc{M})\}$ is  $\bfPi^0_{\alpha+2}$-hard.
		
		In $\Mod(A,R)$ let $\text{Im}(\Phi_\beta)$ be the (closure under isomorphisms) of the image of $\Phi_\beta$. By Theorem \ref{thm:BiInt} we have $\mc{N} \geq_2 \mc{M}$ if and only if $\Phi_\beta(\mc{N}) \geq_{\alpha} \Phi_\beta(\mc{M})$. Thus
		\[ \{\mc{N} : \mc{N} \geq_\alpha \Phi_\beta(\mc{M})\} \cap \text{Im}(\Phi_\beta) = \{ \mc{N}^* : \text{$\mc{N}^* \cong \Phi_\beta(\mc{N})$ for some $\mc{N} \geq_2 \mc{M}$}\}.\]
		By Lemma \ref{lem:contJumpInversion} this latter set is $\bfPi^0_{\beta+4}$ hard, and moreover it is $\bfPi^0_{\beta+4}$ hard within $\text{Im}(\Phi_\beta)$. Thus
		\[ \{\mc{N} : \mc{N} \geq_\alpha \Phi_\beta(\mc{M})\} \]
		is $\bfPi^0_{\beta+4}$-hard or, equivalently, $\bfPi^0_{\alpha+2}$-hard.
	\end{proof}
	
	\begin{lemma}
		For $\alpha=\beta+3$, there is a structure $\mc{M}$ such that the set
		\[ \{\mc{N} : \mc{N} \leq_\alpha \mc{M}\}\]
		is $\bfPi^0_{\alpha+3}$-complete.
	\end{lemma}
	
	\begin{proof}
		Let $\mc{M}$ be the structure from Lemma \ref{lem:3hard} such that $\{\mc{N} : \mc{N} \leq_3 \mc{M}\}$ is  $\bfPi^0_{6}$-complete. Abusing notation, we can assume that $\mc{M}$ is a graph up to effective bi-interpretation. If $\alpha = 3$, we are done; otherwise, if $\alpha =\beta+3>3$, we argue that for the jump inverse $\Phi_\beta(\mc{M})$ we have that  $\{\mc{N} : \mc{N} \leq_\alpha \Phi_\beta(\mc{M})\}$ is  $\bfPi^0_{\alpha+3}$-hard. The argument is the same as in Lemma \ref{cor:abovealpha}.
	\end{proof}
	
	We also generalize the exceptional cases Lemmas \ref{lem:Hard1} and \ref{lemma:2hardeasy} to obtain lower bounds for the remaining exceptional cases close to limit ordinals.
	
	\begin{lemma}
		If $\lambda$ is a limit ordinal there is a structure $\mc{M}$ such that the set
		\[ \{ \mc{N} : \mc{M} \geq_{\lambda+2} \mc{N} \} \]
		is $\bfPi^0_{\lambda+3}$-complete.
	\end{lemma}
	
	\begin{proof}
		Let $\mc{M}$ be structure from Lemma \ref{lemma:2hardeasy} such that $\{\mc{N} : \mc{N} \geq_2 \mc{M}\}$ is  $\bfPi^0_{3}$-complete. Abusing notation, we can assume that $\mc{M}$ is a graph up to effective bi-interpretation. If $\alpha =\lambda+2$ for our limit ordinal $\lambda$, we argue that for the jump inverse $\Phi_\lambda(\mc{M})$ we have that  $\{\mc{N} : \mc{N} \geq_\alpha \Phi_\lambda(\mc{M})\}$ is $\bfPi^0_{\alpha+1}$-hard. This is argued exactly the same as in Lemma \ref{cor:abovealpha}.
	\end{proof}
	
	\begin{lemma}\label{lem:lambdaplusone}
		If $\alpha=\lambda+1$ where $\lambda$ is a limit ordinal, there is a structure $\mc{M}$ such that the set
		\[ \{\mc{N} \; : \; \mc{N} \leq_\alpha \mc{M}\}\]
		is $\bfPi^0_{\alpha+1}$-complete.
	\end{lemma}
	
	\begin{proof}
		Let $\mc{M}$ be the structure from Lemma \ref{lem:Hard1} such that $\{\mc{N} \; : \; \mc{N} \leq_1 \mc{M}\}$ is  $\bfPi^0_{2}$-complete. Abusing notation, we can assume that $\mc{M}$ is a graph up to effective bi-interpretation. If $\alpha =\lambda+1$ for limit ordinal $\lambda$, we argue that for the jump inverse $\Phi_\lambda(\mc{M})$ we have that  $\{\mc{N} \; : \; \mc{N} \leq_\alpha \Phi_\lambda(\mc{M})\}$ is  $\bfPi^0_{\alpha+1}$-hard. This is argued exactly the same as in Lemma \ref{cor:abovealpha}.
	\end{proof}
	
	To generalize Lemma \ref{lem:uniform}, we need slightly more than Lemma \ref{lem:contJumpInversion}, but the idea is essentially the same. The only difference is that we are now considering pairs of structures.
	
	\begin{lemma}
		For $\alpha$ not a limit ordinal, the set of pairs of structures
		\[ \{(\mc{M},\mc{N}) \; : \; \mc{M} \leq_\alpha \mc{N}\}\]
		is $\bfPi^0_{2\alpha}$-complete.
	\end{lemma}
	
	\begin{proof}
		For finite ordinals we have Lemma \ref{lem:uniform}. So assume that $\alpha$ is infinite and write $\alpha=\lambda+n$ for a limit ordinal $\lambda$ and $n \geq 1$. Note that $2\alpha=\lambda+2n$.
		Consider the construction from Lemma \ref{lem:uniform} showing that the set $\{(\mc{M},\mc{N}) \; : \; \mc{M} \leq_n \mc{N}\}$ is $\bfPi^0_{2n}$-complete.
		Abusing notation, we can assume that $\mc{M}$ and $\mc{N}$ are graphs up to effective bi-interpretation. We argue using the $\lambda$-jump inversion that $\{(\mc{M},\mc{N}) \; : \; \mc{M} \leq_\alpha \mc{N}\}$ is  $\bfPi^0_{\lambda+2n}$-hard.
		
		Consider a $\bfPi^0_{\lambda+2n}$ set $A$.
		In a manner analogous to the argument in Lemma \ref{lem:contJumpInversion}, we may refine the topology so that $A$ is $\bfPi^0_{2n}$ and find a map $G$ that witnesses the hardness of $\{(\mc{M},\mc{N}) \; : \; \mc{M} \leq_n \mc{N}\}$ with respect to this map.
		With respect to the original topology, this map $G$ is $\bfSigma^0_{\lambda+1}$-measurable.
		Note that $G(x)$ is a pair $(\mc{M}_x,\mc{N}_x)$.
		The map $G$ factors into two $\bfSigma^0_{\lambda+1}$-measurable maps $G_1(x)=\mc{M}_x$ and  $G_2(x)=\mc{N}_x$.
		Lemma \ref{lem:functoralpha} gives us continuous maps $\Phi_\lambda(G_i)$ such that $\Phi_\lambda(G_i)(x)=\Phi_\lambda(G_i(x))$.
		Note that by Theorem \ref{thm:BiInt},
		\[x\in A \iff G_1(x)\leq_n G_2(x) \iff \Phi_\lambda(G_1(x)) \leq_{\lambda+n} \Phi_\lambda(G_2(x)) \iff \Phi_\lambda(G_1)(x) \leq_{\lambda+n} \Phi_\lambda(G_2)(x).\]
		Therefore, the map $(\Phi_\lambda(G_1),\Phi_\lambda(G_2))$ is a reduction from $A$ to $\{(\mc{M},\mc{N}) \; : \; \mc{M} \leq_\alpha \mc{N}\}$, so the set is $\bfPi^0_{\lambda+2n}$ hard.
	\end{proof}
	
	\section{$\E_\alpha$ and $\A_\alpha$ formulas}\label{sec:defineAE}
	
	As we have seen, the $\Sigma_\alpha$/$\Pi_\alpha$ hierarchy is not optimal for defining back-and-forth types. In particular, for a fixed $\mc{M}$, defining the set
	\[ \{ \mc{N} : \mc{N} \geq_\alpha \mc{M} \}\]
	requires a $\Pi_{\alpha + 2}$ formula, but it is not true that if $\varphi$ is $\Pi_{\alpha+2}$ and $\mc{M} \leq_\alpha \mc{N}$ and $\mc{M} \models \varphi$ then $\mc{N} \models \varphi$. In this section, we define a hierarchy of complexity classes of $\mc{L}_{\omega_1 \omega}$ formulas which more precisely capture the back-and-forth relations. We call these the \textit{back-and-forth complexity classes} of $\mc{L}_{\omega_1 \omega}$ formulas.
	
	\begin{definition}
		We define $\A_\alpha$ and $\E_\alpha$ (and simultaneously ${\vwA_\alpha}$ and ${\vwE_\alpha}$) by mutual transfinite recursion.
		\begin{itemize}
			\item  $\A_1 := \Pi_1$
			\item $\E_1 := \Sigma_1$
			\item  $\A_\alpha :=$ closure of $\bigcup_{\beta < \alpha} \vwE_\beta$ under $\forall$ and $\bigdoublewedge_{i \in \omega}$
			\item  $\E_\alpha :=$ closure of $\bigcup_{\beta < \alpha} \vwA_{\beta}$ under $\exists$ and $\bigdoublevee_{j \in \omega}$
			\item  $\vwE_\alpha :=$ closure of $\E_\alpha$ under $\bigdoublevee_{j \in \omega}, \bigdoublewedge_{i \in \omega}$
			\item  $\vwA_\alpha :=$ closure of $\A_\alpha$ under $\bigdoublevee_{j \in \omega}, \bigdoublewedge_{i \in \omega}$
		\end{itemize}
		Each $\A_\alpha$ formula can be written in the form
		\[ \bigdoublewedge_{i \in I} \forall \bar{y}_i \varphi_i(\bar{x},\bar{y}_i)\]
		where the formulas $\varphi_i$ are $\vwE_\beta$ for some $\beta < \alpha$, and similarly a $\E_\alpha$ formula can be written in the form
		\[ \bigdoublevee_{i \in I} \exists \bar{y}_i \varphi_i(\bar{x},\bar{y}_i)\]
		where each $\varphi_i$ is $\vwA_\beta$ for some $\beta < \alpha$.
	\end{definition}
	\noindent One could also make these definitions for $\mc{L}_{\infty, \omega}$ formulas by allowing the conjunctions and disjunctions to be over arbitrary sets; \ref{prop:transfer-over-bf} still holds in this context. For simplicity we limit ourselves to $\mc{L}_{\omega_1 \omega}$.
	
	The following are some basic properties of these complexity classes that can all be confirmed by straightforward induction arguments. The complexity classes $\exists_\alpha$ and $\forall_\alpha$ are defined by counting alternations of quantifiers but not counting infinite conjunctions and disjunctions.

	\begin{proposition}
		Up to equivalence in countable structures:
		\begin{itemize}
			\item  $\Sigma_\alpha \subseteq \E_\alpha \subseteq\vwE_\alpha \subseteq \exists_\alpha \cap \A_{\alpha+1} \cap\E_{\alpha+2}$
			\item
			$\bigdoublevee \E_\alpha \subseteq \E_\alpha$,
			$\bigdoublewedge \E_\alpha \subseteq \vwE_\alpha$,
			$\exists \E_\alpha \subseteq \E_\alpha$,
			$\forall \E_\alpha \subseteq \A_{\alpha+1}$
			\item
			$\E_\alpha\subseteq \E_{\alpha+1}$,
			$\A_\alpha\subseteq \A_{\alpha+1}$,
			$\vwE_\alpha\subseteq \vwE_{\alpha+1}$,
			$\vwA_\alpha\subseteq \vwA_{\alpha+1}$
			\item
			$\bigdoublevee \vwE_\alpha \subseteq \vwE_\alpha$,
			$\bigdoublewedge \vwE_\alpha \subseteq \vwE_\alpha$,
			$\exists \vwE_\alpha \subseteq \E_{\alpha+2}$,
			$\forall \vwE_\alpha \subseteq \A_{\alpha+1}$
			\item
			$\lnot\E_\alpha=\A_\alpha$,
			$\lnot\A_\alpha=\E_\alpha$,
			$\lnot\vwE_\alpha=\vwA_\alpha$,
			$\lnot\vwA_\alpha=\vwE_\alpha$.
			\item
			for limit $\alpha$: $\vwE_\alpha = \vwA_\alpha =$ closure of $\bigcup_{\beta < \alpha} (\vwE_\beta \cup \vwA_\beta)$ under $\bigdoublevee\bigdoublewedge$, thus $\vwE_\alpha \subseteq \E_{\alpha+1}$ (not $\E_{\alpha+2}$)
			\item
			$\vwE_\alpha \subseteq$ arbitrary $\bigdoublewedge \bigdoublevee$ of $\Sigma_\alpha$
		\end{itemize}
	\end{proposition}
	
	See \cref{fig:hierarchy} for a diagram summarizing the inclusions of these syntactic classes and a diagram from the classical $\Sigma_\alpha$ hierarchy for comparison.
	
	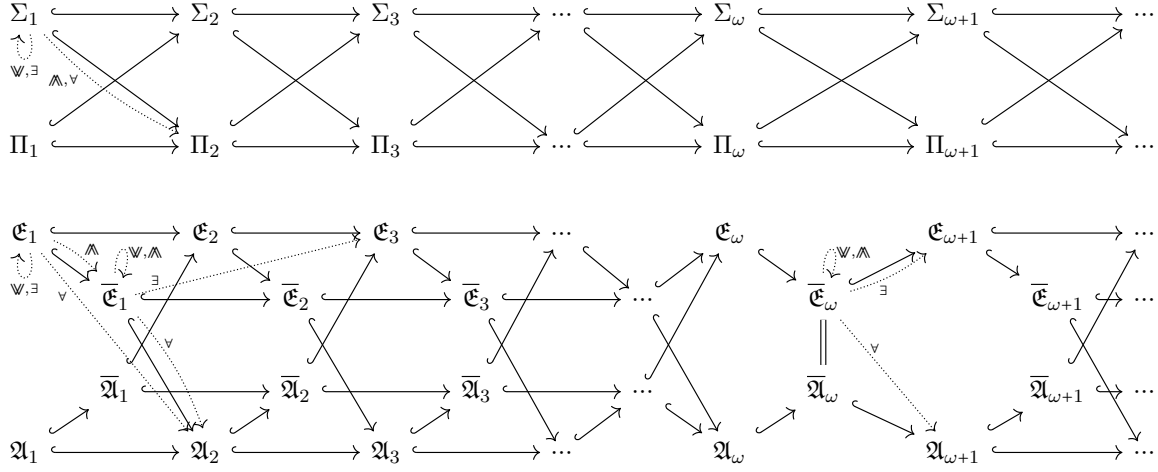
\begin{figure}
		\centering
		\catcode`\&=\active
		\protected\def\-{\relax\ifmmode\expandafter\overline\else\expandafter\hyphen\fi}
		\scalebox{0.88}{\begin{tikzcd}[
				column sep=1.5em,
				operation/.style={
					densely dotted,
					every label/.append style={
						font=\everymath\expandafter{\the\everymath\scriptscriptstyle},
					},
				},
				]
				%% Top diagram: \Sigma_\alpha and \Pi_\alpha
				\Sigma_1 \ar[rr,hook] \ar[ddrr,hook]
				\ar[operation,loop below,"{\bigdoublevee,\exists}"]
				\ar[ddrr,operation,bend right=10,"{\bigdoublewedge,\forall}"'{pos=.3,inner sep=1pt}] &&
				\Sigma_2 \ar[rr,hook] \ar[ddrr,hook] &&
				\Sigma_3 \ar[rr,hook] \ar[ddrr,hook] &&
				\dotsb \ar[rr,hook] \ar[ddrr,hook] &&
				\Sigma_\omega \ar[rr,hook] \ar[ddrr,hook] &&[1em]
				\Sigma_{\omega+1} \ar[rr,hook] \ar[ddrr,hook] &[-1ex]&[-1ex]
				\dotsb
				\\
				\\
				\Pi_1 \ar[rr,hook] \ar[uurr,hook] &&
				\Pi_2 \ar[rr,hook] \ar[uurr,hook] &&
				\Pi_3 \ar[rr,hook] \ar[uurr,hook] &&
				\dotsb \ar[rr,hook] \ar[uurr,hook] &&
				\Pi_\omega \ar[rr,hook] \ar[uurr,hook] &&
				\Pi_{\omega+1} \ar[rr,hook] \ar[uurr,hook] &&
				\dotsb
				\\
				%% Middle diagram: \Exs_\alpha and \All_\alpha
				\Exs_1 \ar[rr,hook] \drar[hook]
				\ar[operation,loop below,"{\bigdoublevee,\exists}"]
				\drar[operation,bend left=10,shift left=1,shorten >=-.5ex,"\bigdoublewedge"{pos=.7,inner sep=1pt}]
				\ar[dddrr,operation,"{\forall}"'{pos=.2,inner sep=1pt}] &&
				\Exs_2 \ar[rr,hook] \drar[hook] &&
				\Exs_3 \ar[rr,hook] \drar[hook] &&
				\dotsb \drar[hook] &&
				\Exs_\omega \drar[hook] &&
				\Exs_{\omega+1} \ar[rr,hook] \drar[hook] &&
				\dotsb
				\\[-1em]
				&
				\-\Exs_1 \ar[rr,hook] \ar[ddr,hook]
				\ar[operation,out=90,in=70,loop,shorten >=-.5ex,"{\scriptscriptstyle\bigdoublevee,\bigdoublewedge}"{right,pos=.7,inner sep=0pt}]
				\ar[ddr,operation,bend left=10,shift left,"\forall"{pos=.3,inner sep=0pt}]
				\ar[urrr,operation,shorten <=-.5ex,"\exists"{pos=.1,inner sep=1pt}] &&
				\-\Exs_2 \ar[rr,hook] \ar[ddr,hook] &&
				\-\Exs_3 \ar[rr,hook] \ar[ddr,hook] &&
				\dotsb \urar[hook] \ar[ddr,hook] &&
				\-\Exs_\omega \dar[equal] \urar[hook]
				\ar[operation,out=90,in=70,loop,shorten >=-.5ex,"{\bigdoublevee,\bigdoublewedge}"{right,pos=.7,inner sep=0pt}]
				\ar[ddr,operation,"\forall"{pos=.3,inner sep=0pt}]
				\urar[operation,bend right=10,"\exists"'{pos=.3,inner sep=1pt}] &&
				\-\Exs_{\omega+1} \rar[hook] \ar[ddr,hook] &
				\dotsb
				\\
				&
				\-\All_1 \ar[rr,hook] \ar[uur,hook] &&
				\-\All_2 \ar[rr,hook] \ar[uur,hook] &&
				\-\All_3 \ar[rr,hook] \ar[uur,hook] &&
				\dotsb \drar[hook] \ar[uur,hook] &&
				\-\All_\omega \drar[hook] &&
				\-\All_{\omega+1} \rar[hook] \ar[uur,hook] &
				\dotsb
				\\[-1em]
				\All_1 \ar[rr,hook] \urar[hook] &&
				\All_2 \ar[rr,hook] \urar[hook] &&
				\All_3 \ar[rr,hook] \urar[hook] &&
				\dotsb \urar[hook] &&
				\All_\omega \urar[hook] &&
				\All_{\omega+1} \ar[rr,hook] \urar[hook] &&
				\dotsb
		\end{tikzcd}}
		\caption{Hierarchies of $\Sigma_\alpha/\Pi_\alpha$ and $\Exs_\alpha/\All_\alpha$ formulas. Solid arrows denote inclusions. The dotted arrows show how syntactic operations affect the complexity of formulas.}
		\label{fig:hierarchy}
	\end{figure}
	
	\medskip
	
	As indicated above, the key properties of this hierarchy lie in the way that it interacts with the $\alpha$ back-and-forth relations.
	We show first that the $\alpha$ back-and-forth relations guarantee agreement on $\vwE_\alpha$ and $\vwA_\alpha$ formulas the same way that they guarantee agreement on $\Sigma_\alpha$ and $\Pi_\alpha$ formulas even though there are more $\vwE_\alpha$ and $\vwA_\alpha$ formulas.
	In particular, we aim to extend the following theorem of Karp \cite{Karp} (see e.g. \cite{MonBook}, Theorem II.36).
	
	\thmKarp*
% 	\begin{theorem}\label{thm:Karp}
% 		For any non-zero ordinal $\alpha$, structures $\mc{M}$ and $\mc{N}$ and tuples $\bar{a}\in\mc{M}$ and $\bar{b}\in\mc{N}$, the following are equivalent:
% 		\begin{enumerate}
% 			\item $(\mc{M},\bar{a})\leq_\alpha (\mc{N},\bar{b})$.
% 			\item Every $\Pi_\alpha$ formula true about $\bar{a}$ in $\mc{M}$ is true about $\bar{b}$ in $\mc{N}$.
% 			\item Every $\Sigma_\alpha$ formula true about $\bar{b}$ in $\mc{N}$ is true about $\bar{a}$ in $\mc{M}$.
% 		\end{enumerate}
% 	\end{theorem}
	
	Given this tight connection between the complexity classes of the hyperarithmetic hierarchy and the back-and-forth game, one may wonder why these classes should not themselves be considered the classes of ``back-and-forth complexity''.
	As it turns out, the analog of Karp's theorem holds with $\Pi_\alpha$ replaced by $\vwA_\alpha$ and $\Sigma_\alpha$ replaced by $\vwE_\alpha$.
	Furthermore, the back-and-forth complexity classes will enjoy many connections with the back-and-forth relations that ordinary $\Pi_\alpha$ and $\Sigma_\alpha$ formulas do not enjoy.
	We begin by demonstrating the analog of Karp's theorem for  $\vwA_\alpha$ and $\vwE_\alpha$.
	Because these classes are bigger than $\Pi_\alpha$ and $\Sigma_\alpha$, it is enough to demonstrate that the back-and-forth relations also preserve all $\vwA_\alpha$ and $\vwE_\alpha$ in the appropriate manner.
	
	\proptransferover*
	
	\begin{proof}
		We argue by induction on complexity. Our base case is $\alpha = 1$ and $\A_1$ and $\E_1$ formulas, which are just $\Pi_1$ and $\Sigma_1$ formulas. For these, the proposition follows directly from Theorem \ref{thm:Karp}.	Now, we have the inductive cases. In the inductive argument, we use the fact that we can group the outside quantifiers in a $\A_\alpha$ or $\E_\alpha$ sentence.
		
		Suppose that $\varphi(\bar{x})$ is $\E_\alpha$, say $\varphi(\bar{x}) = \bigdoublevee_i \exists \bar{y}_i \theta_i(\bar{x},\bar{y}_i)$ where the $\theta_i(\bar{x},\bar{y}_i)$ are $\vwA_\beta$ for some $\beta < \alpha$. Suppose that $\mc{N} \models \varphi(\bar{b})$; then there is $i$ and $\bar{b}'$ such that $\mc{N} \models \theta_i(\bar{b},\bar{b}')$. Then there is $\bar{a}'$ such that $(\mc{N},\bar{b}\bar{b}') \leq_\beta (\mc{M},\bar{a}\bar{a}')$. By the induction hypothesis, $\mc{M} \models \theta_i(\bar{a},\bar{a}')$. Thus $\mc{M} \models \varphi(\bar{a})$ as desired.
		
		Suppose that $\varphi(\bar{x})$ is $\A_\alpha$, say $\varphi(\bar{x}) = \bigdoublewedge_i \forall \bar{y}_i \theta_i(\bar{x},\bar{y}_i)$ where the $\theta_i(\bar{x},\bar{y}_i)$ are $\vwE_\beta$ for some $\beta < \alpha$. This case is dual to the previous one. Suppose that $\mc{M} \models \varphi(\bar{a})$; then for every $\bar{a}'\in\mc{M}$ and $i$,  $\mc{M} \models \theta_i(\bar{a},\bar{a}')$.  Suppose there is some $\bar{b'}\in\mc{N}$ and $i$ such that $\mc{N}\not\models \theta_i(\bar{b},\bar{b}')$. Then there is $\bar{a}'$ such that $(\mc{N},\bar{b}\bar{b}') \leq_\beta (\mc{M},\bar{a}\bar{a}')$. By the induction hypothesis, $\mc{M} \not\models \theta_i(\bar{a},\bar{a}')$, a contradiction. Thus $\mc{N} \models \varphi(\bar{b})$ as desired.
		
		For the case of $\vwE_\alpha$, we induct on the number of infinitary conjunctions or disjunctions in front of the base $\E_\alpha$ formula. Suppose that $\varphi(\bar{x})$ is $\vwE_\alpha$, and say $\varphi(\bar{x}) = \bigdoublevee_i \theta_i(\bar{x})$ where each $\theta_i(\bar{x})$ is $\vwE_\alpha$ with fewer infinitary conjunctions or disjunctions in the front. Suppose that $\mc{N} \models \varphi(\bar{b})$; then for some $i$, $\mc{N} \models \theta_i(\bar{b})$, and so by the induction hypothesis, $\mc{M} \models \theta_i(\bar{a})$. Thus $\mc{M} \models \varphi(\bar{a})$. Similarly, suppose that $\varphi(\bar{x})$ is $\vwE_\alpha$, and say $\varphi(\bar{x}) = \bigdoublewedge_i \theta_i(\bar{x})$ where each $\theta_i(\bar{x})$ is $\vwE_\alpha$. Suppose that $\mc{N} \models \varphi(\bar{b})$; then for all $i$, $\mc{N} \models \theta_i(\bar{b})$, and so by the induction hypothesis, $\mc{M} \models \theta_i(\bar{a})$ for all $i$. Thus $\mc{M} \models \varphi(\bar{a})$.
		
		The cases where $\varphi(\bar{x})$ is $\vwA_\alpha$, and either of the form $\varphi(\bar{x}) = \bigdoublevee_i \theta_i(\bar{x})$ or $\varphi(\bar{x}) = \bigdoublewedge_i \theta_i(\bar{x})$ where each $\theta_i(\bar{x})$ is $\vwA_\alpha$, is identical to the previous case.
	\end{proof}
	
	The following is immediate from the above lemma, along with Karp's theorem.
	
	\begin{corollary}
		For any $\alpha \geq 1$, structures $\mc{M}$ and $\mc{N}$, and tuples $\bar{a}\in\mc{M}$ and $\bar{b}\in\mc{N}$, the following are equivalent:
		\begin{enumerate}
			\item $(\mc{M},\bar{a})\leq_\alpha (\mc{N},\bar{b})$.
			\item Every $\Pi_\alpha$ formula true about $\bar{a}$ in $\mc{M}$ is true about $\bar{b}$ in $\mc{N}$.
			\item Every $\Sigma_\alpha$ formula true about $\bar{b}$ in $\mc{N}$ is true about $\bar{a}$ in $\mc{M}$.
			\item Every $\A_\alpha$ formula true about $\bar{a}$ in $\mc{M}$ is true about $\bar{b}$ in $\mc{N}$.
			\item Every $\E_\alpha$ formula true about $\bar{b}$ in $\mc{N}$ is true about $\bar{a}$ in $\mc{M}$.
			\item Every $\vwA_\alpha$ formula true about $\bar{a}$ in $\mc{M}$ is true about $\bar{b}$ in $\mc{N}$.
			\item Every $\vwE_\alpha$ formula true about $\bar{b}$ in $\mc{N}$ is true about $\bar{a}$ in $\mc{M}$.
		\end{enumerate} 
	\end{corollary}
	
	This means that when it comes to interactions with back-and-forth relations, our new notions of complexity act at least as nicely as the classical notions.
	That said, our new notions have properties that the classical ones do not.
	The key difference is distilled in the following result.
	Classically, there may be no maximal  $\Pi_\alpha$ or $\Sigma_\alpha$ formula, i.e., one which captures the entire $\Pi_\alpha$ or $\Sigma_\alpha$ theory of a given structure or tuple (we give examples of such structures in the following section).
	In the past, a $\Pi_{2\alpha}$ formula was used to describe theories at this level (see \cite{MonBook} Lemma VI.14), though we showed in \cref{lem:aboveGeneral} that $\Pi_{\alpha+2}$ is sufficient.
	Our larger classes contrast with this; they are always able to capture an entire theory at level $\alpha$ with a formula at that same level.
	
	\thmbnfformulas*
	
	\begin{proof}
		We argue by induction, starting with $\alpha = 1$. For $\psi_{\bar{a},\mc{M},1}$, for each $m$ there are only finitely many possible atomic $m$-types (using only the first $m$ atomic formulas). Let $\theta_{\bar{a}}(\bar{x})$ be the finitary quantifier-free formula which says that $\bar{x}$ and $\bar{a}$ satisfy the same atomic formulas (from among the first $|\bar{a}|$-many formulas). In a finite relational language, like that of linear orders, we can take $\theta_{\bar{a}}(\bar{x})$ to say that $\bar{a}$ and $\bar{x}$ have the same atomic type, i.e., are ordered in the same way. Then $(\mc{N},\bar{b}) \geq_1 (\mc{M},\bar{a})$ if and only if
		\[ \mc{N} \models \bigdoublewedge_{n \in \mathbb{N}} \forall \bar{y}^n \bigdoublevee_{\bar{a}' \in \mc{M}^n} \theta_{\bar{a}\bar{a}'}(\bar{b},\bar{y}).\]
		The inner disjunction looks like it is infinite, but in fact, it is not; this is because there are only finitely many possible formulas $\theta_{\bar{a}\bar{a}'}(\bar{x})$. Thus, this formula is $\A_1$ (and in fact $\Pi_1$). For $\varphi_{\bar{a},\mc{M},1}$, we have $(\mc{N},\bar{b}) \leq_1 (\mc{M},\bar{a})$ if and only if
		\[ \mc{N} \models \bigdoublewedge_{\bar{a}' \in \mc{M}} \exists \bar{y}  \theta_{\bar{a}\bar{a}'}(\bar{b},\bar{y}).\]
		This is $\vwE_1$.

		Now, given $\alpha > 1$, suppose that we have $\vwE_\beta$ formulas $\varphi_{\bar{a},\mc{M},\beta}$ and $\A_\beta$ $\psi_{\bar{a},\mc{M},\beta}$ for $\beta < \alpha$. Then $(\mc{N},\bar{b}) \leq_{\alpha} (\mc{M},\bar{a})$ if and only if
		\[ \bigdoublewedge_{\beta < \alpha} \bigdoublewedge_{\bar{a}' \in \mc{M}} \exists \bar{y} \psi_{\bar{a}\bar{a}',\mc{M},\beta}(\bar{b},\bar{y}) \]
		and
		$(\mc{N},\bar{b}) \geq_{\alpha} (\mc{M},\bar{a})$ if and only if
		\[ \bigdoublewedge_{\beta < \alpha} \bigdoublewedge_{m \in \mathbb{N}} \forall \bar{y}^m \bigdoublevee_{\bar{a}' \in \mc{M}} \varphi_{\bar{a}\bar{a}',\mc{M},\beta}(\bar{b},\bar{y}).\]
		These define our desired $\vwE_\alpha$ formulas $\varphi_{\bar{a},\mc{M},\alpha}$ and $\A_\alpha$ formulas $\psi_{\bar{a},\mc{M},\alpha}$.
	\end{proof}
	
	\section{Henkin constructions}
	
	In countable model theory, Henkin constructions with restricted sets of formulas---often in some countable fragment, or in a complexity class such as $\Sigma_\alpha$---have proved useful, especially for type omitting arguments. We will describe two applications of the $\A_\alpha$/$\E_\alpha$ formulas in Henkin constructions.
	
	\subsection{Scott ranks of models of a theory of linear orders}\label{sec:henkin}
	
	In \cite{GHT}, Gonzalez and Harrison-Trainor prove the following theorem:
	
	\begin{theorem}[Gonzalez and Harrison-Trainor \cite{GHT}]\label{thm:ght}
		Given a satisfiable $\Pi_\alpha$ sentence $T$ of linear orders, there is a linear order
		$\mc{B} \models T$ with a $\Pi_{\alpha +4}$ Scott sentence.
	\end{theorem}
	
	The proof is a Henkin construction using formulas of bounded complexity. One of the two key facts about linear orders used in this proof is the well-known characterization of the back-and-forth relations between tuples in linear orders which appears, for example, as Lemma 15.8 of \cite{AK00}.
	
	\begin{lemma}\label{lem:combine-bf}
		Given linear orders $(\mc{A},\bar{a})$ and $(\mc{B},\bar{b})$ (with the tuples in increasing order), $(\mc{A},\bar{a}) \geq_\alpha (\mc{B},\bar{b})$ if and only if for each $i = 0,\ldots,n$ we have $(a_i,a_{i+1})^{\mc{A}} \geq_\alpha (b_i,b_{i+1})^{\mc{B}}$.
	\end{lemma}
	
	In other words, the $\alpha$-back-and-forth type of a tuple ``factors'' into the back-and-forth types of the intervals defined by the tuples.
	That said, this correspondence cannot necessarily be performed on a formula-by-formula basis.
	In other words, given a $\Sigma_\alpha$ formula that holds of $\bar{c}$, it is not necessarily implied by a collection of $\Sigma_\alpha$ formulas about the intervals it defines. It turns out, as proved in \cite{GHT}, that the $\E_\alpha$/$\A_\alpha$ complexity classes are exactly the right complexity classes to make this true. This lemma is provably false when $\E_\alpha$ is replaced by $\Sigma_\alpha$.
	
	\begin{lemma}
		Let $\mc{A}$ be a countable linear order and $a_1<\cdots<a_n$ elements of $\mc{A}$. Suppose that $\mc{A}\models\varphi(a_1,\ldots,a_n)$ with $\varphi$ a $\E_\alpha$ formula in the language of linear orders. Then there are $\E_\alpha$ sentences $\theta_0,\ldots,\theta_n$ such that (a) for every $k = 0,\ldots,n$ we have $(a_k,a_{k+1}) \models \theta_k$, and (b) if $\mc{B}$ is any linear order and $b_1 < \cdots < b_n$, if for every $k = 0,\ldots,n$ we have $(b_k,b_{k+1}) \models \theta_k$ then $\mc{B} \models \varphi(b_1,\ldots,b_n)$. 
	\end{lemma}
	
	\noindent The Henkin construction for proving Theorem \ref{thm:ght} is a Henkin construction where all of the formulas are $\E_{\alpha+1}$.

	\subsection{Omitting types}\label{sec:omittingtypes}
	
	In this section, we will demonstrate how the $\E_\alpha$/$\A_\alpha$ formulas can be used in a Henkin construction by proving Theorem \ref{thmomitting1}. This theorem was already known to Gonzalez (unpublished), but the proof given here is simpler than the original proof.

	Scott's isomorphism theorem states that infinitary sentences can always be used as isomorphism invariants for countable structures \cite{Sco65}.
	Based on this, Montalb\'an defined his notion of robust Scott rank in a way that robustly aligns with the hyperarithmetic hierarchy \cite{MonSR}.
	In this section, we show that his notion of Scott rank also aligns well with our new notion of syntactic complexity.
	This gives a new set of criteria equivalent to Montalb\'an's notion of Scott rank, showing that it is even more robust than we previously thought.
	In particular, the addition of more infinitary conjunctions and disjunctions to our formulas does not increase their expressive power too much when it comes to defining an isomorphism invariant.
	To demonstrate these claims, we perform a type-omitting argument that extends and supersedes that of Montalb\'an from \cite{MonSR}.

	We begin by defining a new syntactic class closely related to the $\E_\alpha$ hierarchy.
	\begin{definition}
		A formula is $\forall\E_\alpha$ if it is of the form  $\bigwwedge_{i\in\omega} \forall \bar{x}_i \phi_i(x_i)$ with $\phi_i\in \E_\alpha$.
	\end{definition}
	
	To be explicit, the difference between $\forall\E_\alpha$ and $\A_{\alpha+1}$ is that $\A_{\alpha+1}$ adds universal quantifiers and conjunctions to $\vwE_\alpha$ formulas while $\forall\E_\alpha$ only adds universal quantifiers and conjunctions to the slightly more restricted class of $\E_\alpha$ formulas.
	We begin with a helpful lemma and then prove a type-omitting theorem for $\forall\E_\alpha$ formulas.
	
	\begin{lemma}\label{lem:internalSigma}
		Given a countable structure $\mc{M}$ and any $\varphi(\bar{x})\in\E_\alpha$ there is a $\Sigma_\alpha$ formula $\theta(\bar{x})$ such that
		\[\mc{M}\models  \forall \bar{x} \big(\theta(\bar{x}) \iff  \varphi(\bar{x})\big).\]
	\end{lemma}
	
	\begin{proof}
		Write the formula $\varphi(\bar{x})$ as $\bigvvee_{i\in \omega} \exists \bar{x}_i \psi_i(\bar{x}_i,\bar{y})$ where each $\psi_i$ is $\vwA_{\beta_i}$ for $\beta_i<\alpha$.
		Consider $S$, the set of witnesses to $\varphi(\bar{x})$ in $\mc{M}$ (note that if $S$ is empty, the statement of the lemma is trivial).
		For each $\bar{a}\in\mc{M}$, one of the disjuncts holds, say $\mc{M}\models \exists \bar{x}_j \psi_{i,\bar{a}}(\bar{x}_j,\bar{a})$.
		Let $\bar{b}_{\bar{a}}$ be a witness giving $\mc{M}\models \psi_i(\bar{b}_{\bar{a}},\bar{a})$.
		By \cite{MonBook} Lemma II.62, there is a $\Pi_{\beta_i}$ formula $\varphi_{\bar{a}}$ that isolates the $\Pi_{\beta_i}$ type of $\bar{b}_{\bar{a}},\bar{a}$ in $\mc{M}$.
		Consider the formula,
		\[\theta(\bar{x}) := \bigvvee_{\bar{a}\in S}  \exists \bar{x} \varphi_{\bar{a}}(\bar{y},\bar{x}).\]
		This is a $\Sigma_\alpha$ formula.
		If $\bar{a}\in S$ then $\mc{M}\models\theta(\bar{a})$ by construction.
		Moreover, if $\mc{M}\models\theta(\bar{c})$ then there is some witness $\bar{a}\in S$ and $\bar{d}\in \mc{M}$ such that $\mc{M}\models\varphi_{\bar{a}}(\bar{d},\bar{c})$.
		This means that $\bar{d},\bar{c}\geq_{\beta_i}\bar{b}_{\bar{a}},\bar{a}$ and so $\mc{M}\models \psi_i(\bar{d},\bar{c}$.
		In particular, this gives that $\bar{c}\in S$, or what is the same, $\mc{M}\models \varphi(\bar{c})$ as desired.
	\end{proof}
	
	We can now prove our type-omitting theorem.
	
	\begin{theorem}\label{thm:typeomit2}
		Let $\mc{M}$ be a countable structure. Let $(\Gamma_i)_{i \in \omega}$ be a list of $\Pi_{\alpha}$-types which are not $\Sigma_{\alpha}$-supported in $\mc{M}$. Let $\chi$ be a $\forall\E_\alpha$ sentence true in $\mc{M}$. Then there is $\mc{N}$ such that $\mc{N} \models \chi$ and $\mc{N}$ omits all of the $\Gamma_i$.
	\end{theorem}
	
	We begin by noting that, by Lemma \ref{lem:internalSigma}, if $\Gamma_i$ is not $\Sigma_{\alpha}$-isolated, then it is not $\E_\alpha$ isolated either. This is essentially the only step of this proof that differs from the proof of Montalbán's type omitting theorem from \cite{MonSR}. The key point is that by expanding the class of formulas considered, the same proof goes through but we obtain Theorem \ref{thmomitting1} as a consequence.
	
	To obtain Theorem \ref{thmomitting1} from Theorem \ref{thm:typeomit2}, we must note that in Proposition \ref{thm:bnfFormulas} we defined, given $\mc{M}$, a $\vwE_\alpha$ sentence $\varphi_{\mc{M},\alpha}$ such that for any structure $\mc{N}$,
	\[ \mc{N} \models \varphi_{\mc{M},\alpha} \Longleftrightarrow \mc{N} \leq_\alpha \mc{M}.\]
	Looking at the proof, $\varphi_{\mc{M},\alpha}$ is actually a countable conjunction of $\E_\alpha$ formulas and hence $\forall \E_\alpha$. Then to obtain Theorem \ref{thmomitting1} we apply Theorem \ref{thm:typeomit2} with  $\varphi_{\mc{M},\alpha}$ as $\chi$. Gonzalez had previously given a proof of Theorem \ref{thmomitting1} which differed significantly from Montalb\'an's; by using the right class of formulas, we obtain a simple proof of this stronger result. We include the full proof of Theorem \ref{thm:typeomit2} for completeness (in part because it does not appear in full in \cite{MonSR}).
	
	\begin{proof}
		Write $\chi$ as $\bigwwedge_{i\in\omega} \forall \bar{x}_i \phi_i(\bar{x}_i)$ with $\phi_i\in \E_\alpha$.
		We build a set of $\E_\alpha$ formulas $T$ over the vocabulary of $\mc{M}$ enriched with countably many constants $C$ to perform a Henkin-like construction.
		We insist on the following properties:
		\begin{enumerate}
			\item If $\bigdoublevee \psi_i \in T$, then for some $i$, $\psi_i \in T$.
			\item If $\exists \bar y \psi(\bar y) \in T$, then $\psi(\bar c) \in T$ for some constants $\bar c \in C$.
			\item If $\bigdoublewedge \psi_i \in T$, then for all $i$, $\psi_i \in T$.
			\item If $\forall \bar y \psi(\bar y) \in T$, then $\psi(\bar c) \in T$ for all $\bar c \in C$.
			\item For every atomic sentence $\psi$ over $\tau \cup C$, either $\psi \in T$ or $\neg \psi \in T$.
			\item For every $i$ and tuple $\bar{c}$ of length $\vert \bar{x}_i \vert$, $\varphi_i(\bar{c})\in T$.
			\item For every tuple $\bar{c}$ of length $\vert \bar{z} \vert$, there is a $\theta\in \Phi(\bar{z})$ such that $\lnot \theta(\bar{c})\in T$.
		\end{enumerate}
		
		Along the way, at each stage $s$, we will make sure that $T$ is satisfiable in the sense that there is some interpretation  $\nu_s:C\to\mc{M}$ that assigns the constants we have used so far to a set of elements that satisfy all of the formulas in $T$ that we have asserted about these constants.
		If we can do this, the claim is shown, as the Henkin model will satisfy $\chi$ by item (6) and omit $\Phi(\bar{z})$ by item (7).
		
		We begin with empty sets $T_0$, $C_0$, and $\nu_0$.
		At each stage we are given $C_s$ a finite subset of $C$, $T_s$, a finite set of  $\E_\alpha$ formulas only mentioning constants $C_s$, and $\nu_s:C_s\to\mc{M}$ with the property that $\mc{M}\models T_s(\nu_s(C_s))$.
		At each stage, we address one of the properties $(1)-(7)$, one instance at a time.
		We describe below how $T_s$, $C_s$, and $\nu_s$ are modified to achieve this.
		
		\begin{enumerate}
			\item Property $(1)$.
			Given $\bigdoublevee \psi_i \in T_s$ we know that $\mc{M}\models \bigdoublevee \psi_i (\nu_s(C_s)))$. 
			Pick a $j$ with the property that $\mc{M}\models \psi_j (\nu_s(C_s)))$ and add $\psi_j$ to $T_s$ to make $T_{s+1}$.
			Let $\nu_s=\nu_{s+1}$ and $C_{s+1}=C_s$.
			It is straightforward to confirm that the desired properties are maintained.
			\item Property $(2)$. 
			Given $\exists \bar y \psi(\bar y) \in T_s$ we know that  $\mc{M}\models  \exists \bar y\psi (\bar{y}, \nu_s(C_s)))$.
			Let $\bar{a}\in\mc{M}$ have the property that $\mc{M}\models \psi(\bar{a},\nu_s(C_s))$.
			Take new constants $\bar{d}\in C$ and let $C_{s+1}=C_s\cup\{\bar{d}\}$ and  $T_{s+1}= T_s\cup\{\psi(\bar{d})\}$.
			Extend $\nu_s$ by defining $\nu_{s+1}(\bar{d})=\bar{a}$.
			It is straightforward to confirm that the desired properties are maintained.
			\item Property $(3)$.
			Given $\bigdoublewedge \psi_i \in T_s$ and $i\in\omega$ let $T_{s+1}=T_s\cup\{\psi_i\}$, $C_{s+1}=C_s$ and $\nu_{s+1}=\nu_s$.
			It is straightforward to confirm that the desired properties are maintained.
			\item Property $(4)$.
			Given $\forall \bar y \psi(\bar y) \in T_s$ and $\bar{c}\in C_s$ of the appropriate length $T_{s+1}=T_s\cup\{\psi(\bar{c})\}$, $C_{s+1}=C_s$ and $\nu_{s+1}=\nu_s$.
			It is straightforward to confirm that the desired properties are maintained.
			\item Property $(5)$.
			For any atomic sentence $\psi$ and $\bar{c}\in C_s$ of the appropriate length check if $\mc{M}\models \psi(\nu_s(\bar{c}))$ or if $\mc{M}\models \lnot\psi(\nu_s(\bar{c}))$.
			In the former case, add $\psi(\bar{c})$ to $T_s$ to obtain $T_{s+1}$ and in the later case  add $\lnot\psi(\bar{c})$ to $T_s$ to obtain $T_{s+1}$.
			Either way, let $C_{s+1}=C_s$ and $\nu_{s+1}=\nu_s$.
			It is straightforward to confirm that the desired properties are maintained.
			\item Property $(6)$.
			Given $\varphi_i$ and $\bar{c}\in C_s$ of the appropriate length $T_{s+1}=T_s\cup\{\varphi_i(\bar{c})\}$, $C_{s+1}=C_s$ and $\nu_{s+1}=\nu_s$.
			It is straightforward to confirm that the desired properties are maintained.
			\item Property $(7)$.
			This is the key property of the type omitting argument and the only property that necessitates shifting $\nu_s$ from its already established action to obtain $\nu_{s+1}$.
			This step follows the standard type omitting playbook in a Henkin construction.
			Fix a tuple  $\bar{c}\in C_s$ of length $\vert \bar{z} \vert$ and let $\bar{d}$ be the elements of $C_s$ that are not among $\bar{c}$.
			We can write $\bigwedge T_s=\rho(\bar{c},\bar{d})$ where $\rho$ is a $\E_\alpha$ formula.
			Note that $\exists\bar{y}\rho(\bar{c},\bar{y})$ is also a $\E_\alpha$ formula that is satisfied in $\mc{M}$ by $\nu_s(\bar{c})$.
			In particular, because $\Phi(\bar{z})$ is not $\E_\alpha$ supported, there is some $\bar{a}\in \mc{M}$ that satisfies $\exists\bar{y}\rho(\bar{a},\bar{y})$ but does not have type $\Phi(\bar{z})$.
			Let $C_{s+1}=C_s$ and let $T_{s+1}=T_s\cup\{\lnot\theta(\bar{c})\}$ where $\theta\in \Phi(\bar{z})$ has the property that $\mc{M}\models \lnot\theta(\bar{a})$.
			Note that $\lnot\theta\in \E_\alpha$ as $\theta\in \A_\alpha$, so this is an allowable extension of $T$.
			Finally let $\nu_{s+1}(\bar{c})=\bar{a}$ and let $\nu_{s+1}(\bar{d})=\bar{b}$ where $\mc{M}\models \rho(\bar{a},\bar{b})$.
			It is straightforward to confirm that the desired properties are maintained.
		\end{enumerate}
		
		With this, we can achieve a model with properties $(1)-(7)$, and therefore the proof is complete.
	\end{proof}

	We now show that having a $\forall\E_\alpha$ Scott sentence is equivalent to having one that is $\Pi_{\alpha+1}$.
	This adds to a long list of equivalent conditions for having $\SR(\mc{M})\leq \alpha$ (see Theorem 1.1 of \cite{MonSR}). Parts of the proof below, in particular $(1)\implies(2)$ and $(2)\implies(3)$, are very similar to the corresponding parts of the proof of that theorem.

	\begin{theorem}\label{thm:moreRobust}
		The following are equivalent when the classes $\A$ and $\E$ are restricted to $\Lomom$:
		\begin{enumerate}
			\item $\mc{M}$ has a $\forall\E_\alpha$ Scott sentence.
			\item The $\A_\alpha$ type of every tuple in $\mc{M}$ is supported by an $\E_\alpha$ formula.
			\item Every automorphism orbit of every tuple in $\mc{M}$ is definable by a $\E_\alpha$ formula.
			\item $\mc{M}$ has a $\Pi_{\alpha + 1}$ Scott sentence, i.e., $\SR(\mc{M})\leq \alpha$.
		\end{enumerate}
	\end{theorem}
	
	\begin{proof}

		We begin by showing that $(1)\implies(2)$.
		Say that $\mc{M}$ has a $\A_\alpha$ type $\Phi(\bar{z})$ that is not supported by any $\E_\alpha$ formula.
		Note that an  $\A_\alpha$ type is determined by its $\Pi_\alpha$ restriction by Proposition \ref{prop:transfer-over-bf}.
		Also, note that $\Phi(\bar{z})$ is not supported by any $\Sigma_\alpha$ formula.
		By Theorem \ref{thm:typeomit2} any $\chi\in \forall\E_\alpha$ is also true of a model $\mc{N}$ that omits $\Phi$ and therefore is not isomorphic to $\mc{N}$.

		We now move to prove that $(2)\implies(3)$.
		Let $\varphi_{\bar{a}}(\bar{x})$ be the $\E_\alpha$ formula that supports the $\A_\alpha$ type of $\bar{a}$.
		We claim that $\varphi_{\bar{a}}(\bar{x})$ defines the automorphism orbit of $\bar{a}$.
		To demonstrate this, we claim that the set
		$T:=\{(\bar{a},\bar{b}) \vert \mc{M}\models \varphi_{\bar{a}}(\bar{b})\}$
		is a back-and-forth set.
		As $\varphi_{\bar{a}}$ supports the whole $\A_\alpha$ type of $\bar{a}$ it is certainly the case that for $(\bar{a},\bar{b})\in T$, $\bar{a}$ and $\bar{b}$ have the same atomic diagram.
		Now consider $c\in\mc{M}$; we find a $d\in\mc{M}$ such that $\mc{M}\models \varphi_{\bar{a},c}(\bar{b},d)$.
		In other words, we must show that $\mc{M}\models \exists x  \varphi_{\bar{a},c}(\bar{b},x)$.
		If this does not hold, we have that $\mc{M}\models \lnot\exists x  \varphi_{\bar{a},c}(\bar{b},x)$.
		Note that this is a $\A_\alpha$ formula, therefore, it is implied by $\varphi_{\bar{a}}(\bar{b})$ which isolates the $\A_\alpha$ type of $\bar{b}$.
		However, this means that $\mc{M}\models \lnot\exists x  \varphi_{\bar{a},c}(\bar{a},x)$ as $\mc{M}\models \varphi_{\bar{a}}(\bar{a})$.
		That said,  $\mc{M}\models \exists x  \varphi_{\bar{a},c}(\bar{a},x)$ as witnessed by $c$, a contradiction.
		Therefore, there is such a $d$ as required.
		
		To see the ``forth'' part of the argument, it is enough to show that $\mc{M}\models\varphi_{\bar{a}}(\bar{b}) \iff \varphi_{\bar{b}}(\bar{a})$ and appeal to a symmetrical argument as the one above.
		Say that $\mc{M}\models \varphi_{\bar{a}}(\bar{b}) \land  \lnot\varphi_{\bar{b}}(\bar{a})$.
		As $ \lnot\varphi_{\bar{b}}$ is $\A_\alpha$, we obtain that it is implied by $\varphi_{\bar{a}}$ which supports the whole $\A_\alpha$ type in $\mc{M}$.
		This is a contradiction as it yields $\mc{M}\models  \lnot \varphi_{\bar{b}}(\bar{b})$.
		Therefore, $T$ is a back-and-forth set and any $\bar{b}$ with $\mc{M}\models \varphi_{\bar{a}}(\bar{b})$ is automorphic to $\bar{a}$.

		We now show that $(3)\implies(4)$.
		This is just the same as showing that there is a $\Sigma_\alpha$ description of any automorphism orbit in $\mc{M}$ if there is a $\E_\alpha$ description of any automorphism orbit in $\mc{M}$.
		This follows immediately from Lemma \ref{lem:internalSigma}.

		Lastly, we note that $(4)\implies (1)$.
		If $\SR(\mc{M})\leq \alpha$ then $\mc{M}$ has a $\Pi_{\alpha+1}$ Scott sentence.
		As every $\Pi_{\alpha+1}$ formula is $\forall\E_\alpha$, it follows that $\mc{M}$ also has a $\forall\E_\alpha$ Scott sentence.
	\end{proof}
	
	\cor*
	
	\begin{proof}
		For (1), $\mc{M}\leq_{\alpha} \mc{N} \implies \mc{M}\cong \mc{N}$. There is a $\Pi_{\alpha+2}$ sentence defining the set of $\mc{N}$ such that $\mc{M}\leq_{\alpha} \mc{N}$. By assumption, this is a Scott sentence.
		
		For (2), suppose that $\mc{N} \leq_{\alpha} \mc{M} \implies \mc{M}\cong \mc{N}$. Then, by Theorem \ref{thmomitting1}, there is a structure $\mc{N} \leq_\alpha \mc{M}$ such that $\mc{N}$ omits every $\Pi_\alpha$ type which is not $\Sigma_\alpha$-supported in $\mc{M}$. Thus $\mc{N} \cong \mc{M}$ and every $\Pi_\alpha$ type realized in $\mc{M}$ is $\Sigma_\alpha$-supported in $\mc{M}$. Thus $\mc{M}$ has a $\Pi_{\alpha+1}$ Scott sentence.
	\end{proof}
	
	\bibliographystyle{alpha}
	\bibliography{references}
	
\end{document}